\documentclass{svjour3}                   
\smartqed  
\usepackage{amsfonts,mathtools}
\usepackage{psfrag}
\usepackage[table]{xcolor}
\usepackage{mathptmx}      
\DeclareMathAlphabet{\mathcal}{OMS}{cmsy}{m}{n} 

\usepackage{latexsym,amsbsy,amssymb,amsmath}
\usepackage{enumitem}
\usepackage{epstopdf}
\usepackage{mathrsfs}
\usepackage{color,graphicx}
\usepackage{color}
\usepackage{geometry}
\usepackage{epstopdf}
\usepackage{multirow}
\usepackage[justification=centering]{caption}
\usepackage{subcaption, algorithmicx}
\usepackage[]{algorithm}
\usepackage{longtable}
\usepackage{tabularx}
\usepackage{xfrac}
\allowdisplaybreaks

\usepackage{bm}
\usepackage{amsopn}
\DeclareMathOperator{\diag}{diag}
\usepackage{epstopdf}
\usepackage{graphics}
\usepackage{epsf}
\usepackage{amscd}
\usepackage{wrapfig}
\usepackage[mathscr]{euscript}
\usepackage[english]{babel}
\usepackage{epsfig}
\usepackage{dsfont}
\numberwithin{equation}{section}
\newcommand{\bx}{\pmb{x}}
\newcommand{\bq}{\pmb{q}}
\newcommand{\bv}{\pmb{v}}
\newcommand{\bmf}{\pmb{f}}

\newcommand{\bA}{\pmb{A}}
\newcommand{\bD}{\pmb{D}}
\newcommand{\bI}{\pmb{I}}
\newcommand{\bW}{\pmb{W}}
\newcommand{\bR}{\pmb{R}}

\newcommand{\Rv}[1]{\textcolor{black}{#1}}

\journalname{BIT}

\begin{document}

\title{Low regularity integrators for semilinear parabolic equations with maximum bound principles}

\titlerunning{Low regularity integrators for semilinear parabolic equations }

\author{
Cao-Kha Doan \and
Thi-Thao-Phuong Hoang \and
Lili Ju \and
Katharina Schratz
}

\institute{
C.-K. Doan \at
Department of Mathematics and Statistics, Auburn University, Auburn, AL 36849, USA \\
\email{kcd0030@auburn.edu}           
\and
T.-T.-P. Hoang \at
Department of Mathematics and Statistics, Auburn University, Auburn, AL 36849, USA \\
\email{tzh0059@auburn.edu}
\and
L.  Ju \at
Department of Mathematics, University of South Carolina, Columbia, SC 29208, USA \\
\email{ju@math.sc.edu}
\and
K. Schratz \at
LJLL (UMR 7598), Sorbonne Universit\'e, UPMC, 4 place Jussieu, 75005, Paris, France\\
\email{katharina.schratz@ljll.math.upmc.fr}
}

\date{Received: date / Accepted: date}

\maketitle

\begin{abstract}
This paper is concerned with conditionally structure-preserving,  low regularity time integration methods for a class of semilinear parabolic equations of Allen-Cahn type.  Important properties of such equations include maximum bound principle (MBP) and energy dissipation law; for the former,  that means the absolute value of the solution is pointwisely bounded for all the time by some constant imposed by appropriate initial and boundary conditions.  The model equation is first discretized in space by the central finite difference,  then by iteratively using Duhamel's formula,  first- and second-order low regularity integrators (LRIs) are constructed  for time discretization of the semi-discrete system. The proposed LRI schemes are proved to preserve the MBP and the energy stability in the discrete sense. Furthermore,  their temporal error estimates are also successfully derived under a low regularity requirement that the exact  solution of the semi-discrete problem is only assumed to be continuous in time. Numerical results show that the proposed LRI schemes are more accurate and have better convergence rates than classic exponential time differencing schemes,  especially when the interfacial parameter approaches zero.


\end{abstract}


\section{Introduction}

Phase-field modeling is an important mathematical tool for solving interfacial problems in scientific and engineering fields.  Phase-field models were first introduced by Fix~\cite{phasefieldmodel1} and Langer~\cite{phasefieldmodel2}, and have gained increasing interest and attention in solidification dynamics \cite{solidification,solidification2} and other areas such as fracture mechanics~\cite{fracture2,fracture3}, viscous fingering~\cite{viscosity,viscosity2}, hydrogen embrittlement~\cite{hydrogen,hydrogen2}, vesicle dynamics~\cite{vesicle,vesicle2}, etc.  A thorough review on phase-field modeling can be found in~\cite{Fengreview}.  The Allen-Cahn equation, originally introduced by Allen and Cahn in \cite{1979}, is a well-known phase-field model used to describe the motion of antiphase boundaries in crystalline solids.  
It belongs to a class of semilinear parabolic equations of the form: \vspace{-0.1cm}
\begin{equation}\label{Allen-Cahn}
\begin{array}{rll}
\partial_t u&=\varepsilon^2\Delta u + f(u),  & \bx\in \Omega, \; t>0, 
\end{array} 
\end{equation}
where $\Omega \subset \mathbb{R}^d \, (d=1,2,3)$ is an open and bounded Lipschitz domain, $u:\overline{\Omega}\times [0,\infty)\rightarrow \mathbb{R}$ is the unknown function,  $\varepsilon>0$ is the interfacial parameter representing the width of the transition regions,  and $f:\mathbb{R}\rightarrow \mathbb{R}$ is a nonlinear function.  An initial boundary value problem associated with equation~\eqref{Allen-Cahn} is obtained by imposing the initial condition 
$u(\bx,0)=u_0(\bx)$ for any $\bx\in\overline{\Omega}$,
and suitable boundary conditions (e.g., periodic or homogeneous Neumann or Dirichlet boundary conditions).

Suppose that the nonlinear function $f$ satisfies the following assumptions:\vspace{4pt}\\
\noindent (A1) {\it $f(\beta)\le 0 \le f(-\beta)$ for some $\beta>0$.} \vspace{2pt}\\
\noindent (A2) {\it $f$ is continuously differentiable and nonconstant on $[-\beta,\beta].$}\vspace{4pt}\\
\noindent Then the solution to \eqref{Allen-Cahn} satisfies the maximum bound principle (MBP) as analyzed in~\cite{Review}, namely if  the initial data is  pointwisely bounded  in  absolute value by $\beta$,  then the solution is also pointwisely bounded  in  absolute value by $\beta$  for all the time, i.e., \vspace{-0.1cm}
\begin{equation}
 \max_{\bx\in\overline{\Omega}} \vert u_{0} (\bx) \vert \leq \beta \quad \Longrightarrow \quad \max_{\bx\in\overline{\Omega}} \vert u (\bx,t) \vert \leq \beta\quad \text{ for all }\thickspace t >0.
\end{equation}
For example, for the Allen-Cahn equation with $f(u)=u-u^3$~\cite{1992}, the solution is bounded by $\beta=1$.  
In addition, as a phase-field type model,  equation \eqref{Allen-Cahn} can be viewed as an $L^2$ gradient flow with respect to the energy functional \vspace{-0.1cm}
\begin{align}\label{energy functional}
 E(u)=\int_{\Omega}\left(\frac{\varepsilon^2}{2}|\nabla u(\bx)|^2+F(u(\bx))\right)d\bx,   
\end{align}
where $F$ is a smooth potential function with $F'(u)=-f(u)$.  Thus, the solution to~\eqref{Allen-Cahn} decreases the energy~\eqref{energy functional} in time:
\begin{equation}
\frac{d}{dt} E(u(t)) =-\int_{\Omega}|\partial_tu(\bx,t)|^{2} \,d\bx\leq 0.
\end{equation}
Such a property is known as the energy dissipation law of the phase transition process.  MBP and energy dissipation are two physical properties of Allen-Cahn type equations~\eqref{Allen-Cahn}, thus they are highly desired to be preserved when constructing numerical schemes for these equations.

Various MBP-preserving schemes for semilinear parabolic equations have been discovered recently. In~\cite{mainpaper}, it was shown that implicit-explicit discretization in time and  central finite difference in space for the Allen-Cahn equation preserve the MBP. The decay of the discrete energy functional along the time is obtained based on the MBP result.  The generalized Allen-Cahn equation with a nonlinear degenerated mobility was studied in~\cite{genAC},  where a first-order semi-implicit scheme was proved to preserve the MBP under some time step constraint (which can be further improved by adding a stabilizing term).  Following these ideas,  both linear and nonlinear second-order Crank-Nicolson/Adams-Bashforth schemes were constructed in \cite{Hou1,Hou2},  which guarantee the MBP and  energy stability for the Allen-Cahn equation as well as its variants.  When solving problems with strong stiff linear terms,  exponential-type integrators such as exponential time differencing (ETD) and integrating factor Runge-Kutta (IFRK) methods have been shown to be more effective.  Based on Duhamel's formula (also known as the variation-of-constants formula) along with approximating the nonlinear term by polynomial interpolations, the ETD schemes~\cite{Luan13,Luan14} evaluate the linear part exactly, ensuring both stability and accuracy.  In \cite{Review}, a general framework of ETD schemes for a class of semilinear parabolic equations was established, where linear and nonlinear operators have to satisfy certain conditions to preserve the MBP. ETD methods were also used for the nonlocal Allen-Cahn equation \cite{nonlocal} and the conservative Allen-Cahn equation~\cite{conserAC}.  A combination of ETD with overlapping domain decomposition was studied in~\cite{existencetheorem}. 
Unlike ETD, the main idea of IFRK methods is to use the integrating factor to eliminate the stiff linear term before applying the conventional Runge-Kutta method. A four-stage, fourth-order IFRK scheme was proposed in \cite{boundeA} which preserves the MBP under some certain condition on the time step size.  By adding a stabilizing constant,  the stabilized IFRK method was introduced in \cite{IFRK} and shown to unconditionally preserve the MBP up to the third order.  \Rv{In addition, uniform $L^{p}-$ bounds, which are weaker than the MBP,  were also studied for numerical discretization of the Allen-Cahn equation in \cite{Lpnorm}.}

Similar to the MBP preserving schemes,  energy stable numerical methods for time integration of the Allen-Cahn type equations as well as other gradient flows have attracted much attention.  Examples of traditional approaches include fully implicit method \cite{implicit1,implicit2}, convex splitting schemes \cite{splitting1,splitting2,splitting3,splitting4}, stabilized semi-implicit methods \cite{semi-im1,mainpaper,semi-im3,semi-im4} and  ETD schemes \cite{ETD_ener1,ETD_ener2,ETD_ener3,Review,nonlocal}.  Some novel linear schemes have been investigated in the past few years such as the invariant energy quadratization (IEQ) schemes~\cite{IEQ1,IEQ2,IEQ3} and the scalar auxiliary variable (SAV) schemes~\cite{SAV1,SAV2,SAV4}, which monotonically decrease certain modified energies. The main idea of these two methods is to transform the energy functional~\eqref{energy functional} into quadratic forms of some new auxiliary variables. Because of the efficiency of the SAV approach in term of computations, many of its variants were also developed~\cite{v_SAV4,v_SAV5,v_SAV6,v_SAV7}. 

It should be noted that convergence and error estimates play an essential role when analyzing a numerical scheme. A common way is to first discretize the model problem (cf.  equation~\eqref{Allen-Cahn}) in space, then apply time integration to the space-discrete problem.  Convergence of the numerical solution to the exact space-discrete solution is usually obtained under some assumptions on the smoothness of the exact solution.  In particular, ETD schemes \cite{Review,nonlocal} require the exact space-discrete solution to be in $C^1$ for the first-order scheme and in $C^2$ for the second-order scheme. Similarly, convergence of the IFRK schemes \cite{boundeA,IFRK} only holds when the exact solution at least belongs to $C^p$ in time, where $p$ is the order of the method. Therefore, it is desirable to develop time integration schemes for Allen-Cahn type equations which only require low regularity on the exact solution but hold similar optimal  error estimates, while still preserving the two intrinsic properties, namely MBP and energy stability.  

The LRIs are recently introduced time discretization techniques for evolution problems, which require weaker regularity assumptions for error analysis than those by classical methods.  The so-called resonance based schemes were developed for nonlinear dispersive equations at low regularity,  including the Korteweg-de Vries (KdV)~\cite{KdV1,KdV2,KdV3},   Schr\"{o}dinger~\cite{Sch1,Sch2,Sch3},  Boussinesq~\cite{Bou},  and Dirac~\cite{Dirac} equations. These schemes are based on Fourier series expansions of the solution (instead of Taylor expansions as in classical methods),  thus they are restricted to periodic boundary conditions only.  High-order resonance-based methods can be constructed using decorated tree series as shown in \cite{resonance}.  Recently,  a new framework of LRIs has been introduced in \cite{LRIgeneralframework} which does not rely on Fourier expansions, thus can be applied to general boundary conditions. The idea is to introduce filter oscillations to treat the dominant oscillations exactly and use a stabilized Taylor series expansion to approximate the lower order parts.  First- and second-order LRIs have been constructed and analyzed in~\cite{LRIgeneralframework}.  The LRIs can be coupled with various spatial discretization techniques, such as finite element methods as considered in \cite{LRINS}, where a first-order semi-implicit LRI scheme was applied to solve the Navier-Stokes (NS) equations.  Analysis of temporal convergence for both space-continuous and fully discrete problems were carried out under a weaker regularity condition than the semi-implicit Euler method and the classical exponential integrator.  

In this paper, we aim to construct first- and second-order LRIs for the time approximation of equation~\eqref{Allen-Cahn} which is discretized in space by the standard central finite difference.  The proposed schemes are obtained by iterating Duhamel's formula as in~\cite{LRIgeneralframework} and are shown to preserve MBP,  from which energy stability and error estimates of the numerical solutions can be proved independently.  In particular, convergence of the fully discrete numerical solution to the exact space-discrete solution is obtained for both first- and second-order schemes, by only assuming that the exact solution is continuous in time  (instead of $C^{1}$ or $C^{2}$ as required by ETD or IFRK methods).  The main contribution of our work can be summarized as follows. Firstly, we do not require high regularity of the solution to the semi-discrete system of \eqref{Allen-Cahn}
to carry out error estimates, the regularity is passed to the nonlinear  function $f$ instead.  Secondly, we optimize the condition on the time step size $\Delta t$ to preserve the MBP, allowing the use of fairly large $\Delta t$.  Thirdly,  the proofs of energy stability rely only on the MBP,  avoiding some extra conditions on the exact solution (needed for  many other methods due to the use of error estimate results for the proof). Finally, numerical results show that the proposed LRIs (especially the second-order scheme) are more accurate and have better convergence rates than the ETD schemes,  especially when the interface parameter  $\varepsilon$ in \eqref{Allen-Cahn} is very small.  We also would like to note that
 the proposed LRI schemes can be straightforwardly applied to other similar systems,  e.g., nonlocal or fractional Allen-Cahn equations~\cite{nonlocal}.

The rest of this paper is organized as follows. 
In Section \ref{section_schemes}, we discretize the model equation~\eqref{Allen-Cahn}  in space using the central finite difference method and  then derive first- and second-order LRI schemes for time integration of the  space-discrete problem.  Next,  MBP and energy stability of the resulting fully discrete numerical solution  are proved in Section \ref{section_properties}; we also discuss the MBP results  under some typical choices of the potential function.  In Section \ref{section_error_estimates}, we carry out convergence and temporal error analysis of the proposed LRI schemes under low regularity assumptions. In Section \ref{section_numerical}, numerical results are presented to illustrate the accuracy and convergence of the proposed schemes. Finally, some concluding remarks are given in Section \ref{section_conclusion}. 
Throughout the paper, we suppose that assumptions (A1) and (A2) hold; any further conditions on $f$, if required,  will be stated explicitly. 

\section{Space-discrete problem and low regularity integrators}\label{section_schemes}

Assume that the spatial domain $\Omega$ in~$\mathbb{R}^d $ is either a finite interval ($d=1$),  a rectangle ($d=2$), or a rectangular parallelepiped  ($d=3$).  We consider a uniform partition of $\Omega$ into rectangular elements; for simplicity,  we assume all elements have equal sides and denote by $h$ the length of the side of each element.  Let $N$ be the number of grid points of the mesh for unknowns,  and denote by $$\bm{u}(t)=(u^1(t),u^2(t),\ldots, u^N(t))^T,$$
where $u^j(t)$ is the approximation of $u(\bx_j,t)$ for $j=1,2,\ldots,N$.  Using central finite differences for spatial discretization of problem~\eqref{Allen-Cahn} with the homogeneous Neumann boundary conditions,  we obtain the following space-discrete system: 
\begin{align}\label{semi-discrete}
 \frac{d\bm u}{dt} = \bA\bm u+\bmf(\bm u),   
\end{align}
where $\bmf(\bm u)=(f(u^1),f(u^2),\ldots, f(u^N))^T$,  $\bA=\varepsilon^2\bD_h$, and $\bD_h$ is defined as in \cite{fEIF}. 
Particularly,  $\bD_h=\bm\lambda_h$ in one dimension,  $\bD_h=\bI \otimes \bm\lambda_h + \bm\lambda_h\otimes \bI$ in two dimensions,  and $\bD_h=\bI\otimes \bI\otimes \bm\lambda_h + \bI\otimes \bm\lambda_h \otimes \bI +\bm\lambda_h\otimes \bI\otimes I$ in three dimensions. Here $\otimes$ denotes the Kronecker product, $\bI$ is the identity matrix, and $\bm\lambda_h$ is a square matrix of order $N$ given by
$$\bm\lambda_h:=\frac{1}{h^2}\begin{bmatrix}
-2&2&0&\ldots&0\\
1&-2&1&\ldots&0\\
\vdots & \ddots & \ddots & \ddots &\vdots\\
0&\ldots&1& -2 &1\\
0&\ldots&0&2&-2
\end{bmatrix}.$$
It is well-known that matrix $\bD_{h}$ is  diagonally dominant with all diagonal entries negative.  We then have the following result~\cite[Lemma 3.3]{boundeA},  which is crucial to the analysis of exponential-type integrators:
\begin{lemma}\label{bode}
For any $\gamma>0$, we have $\|e^{\gamma\bD_h}\|_{\infty} \le 1,$ and therefore
$\|e^{\gamma\bA}\|_{\infty}\le 1.$
\end{lemma}
Based on this result,  the space-discrete problem~\eqref{semi-discrete} is well-posed and preserves the MBP as stated below (the proof is similar to Theorem 2 in \cite{existencetheorem}). 
\begin{theorem}\label{MBP_exact_solution} Given any fixed terminal time $T > 0$.
Assume that the initial data is bounded by $\beta$ in absolute value,  i.e., $\|u_0\|_{L^{\infty}}\le \beta$,  then the semi-discrete system \eqref{semi-discrete} admits a unique solution $\bm u\in \mathcal{C}([0,T];\mathbb{R}^N)$ and 
$$\|\bm u(t)\|_{\infty}\le \beta \quad \text{ for all } \thickspace t\in [0,T]. $$
\end{theorem}

\Rv{
\begin{remark}
The result of Theorem \ref{MBP_exact_solution} still holds  when one imposes the periodic boundary condition or the Dirichlet boundary condition (where the Dirichlet data is also assumed to be bounded by $\beta$ in absolute value); see \cite{Review} for more details.  In addition,  other types of spatial discretizations (e.g., finite volume and  lumped-mass finite element) for problem~\eqref{Allen-Cahn} can also be considered without affecting the analysis presented in this paper,  provided that the corresponding matrix $\bA=(a_{ij})_{N\times N}$ satisfies the following property~\cite[Proposition 2.5]{Review}:
\begin{equation} \label{eq:Acond}
\vert a_{ii} \vert \geq \sum_{j=1, j\neq i}^{N} \vert a_{ij}\vert, \quad a_{ii} <0, \quad a_{ij} \geq 0 \; (j \neq i).
\end{equation}
A slightly more relaxed condition on $\bA$ for Lemma 1 to hold can be further  found in \cite[Lemma 3.3]{boundeA}.
\end{remark}}

For the time discretization,  let us consider a uniform partition of the time interval $[0,T]: 0=t_0<t_1<\cdots< t_M=T$, with the step size $\Delta t = \frac TM$. The exact (in time) solution to \eqref{semi-discrete} at each time level is given by Duhamel's formula:
\begin{align}
    \bm u(t_{m+1})= e^{\Delta t \bA}\bm u(t_m) + \int_0^{\Delta t} e^{(\Delta t -s)\bA} \bmf(\bm u(t_m+s))\,ds,\;\; m=0,1,\ldots,M-1.  \label{exact1}
\end{align}
The main idea of LRIs developed in \cite{LRIgeneralframework,LRINS} is to use again Duhamel's formula to compute $\bm u(t_m+s)$ in \eqref{exact1}.  In particular, we have,  for any $s\in [0,\Delta t]$:
\begin{align}
    \bm u(t_m+s)&= e^{s\bA}\bm u(t_m)+\int_0^s e^{(s-\sigma)\bA} \bmf(\bm u(t_m+\sigma))\,d\sigma. \label{exact2}
\end{align}
Using \eqref{exact2} to approximate $\bm u(t_m+s)$ in \eqref{exact1} leads to the first- and second-order LRIs described below. 

For convenience of presentation, we introduce  the following notation: for any function $g:\mathbb{R}\to\mathbb{R}$,  let $\bm g: \mathbb{R}^{N}\to\mathbb{R}^{N}$ be defined as $\bm g(\bm v):=(g(v_1),g(v_2),\ldots,g(v_N))^T$, for any $\bm v=(v_1,v_2,\ldots,v_N)^T\in\mathbb{R}^N$; when $g$ is sufficiently smooth, we denote by
\begin{equation} \label{eq:deri_notation}
\frac{\partial^p\bm g}{\partial \bm v^p}(\bm v):=\diag(g^{(p)}(v_1),g^{(p)}(v_2),\ldots,g^{(p)}(v_N)),\quad p\in \mathbb{N}.
\end{equation}
Note that when $p=1$,  \eqref{eq:deri_notation} is simply the Jacobian of $\bm g$. In addition, we make use of the notation~$\odot$ for  denoting element-wise operations, e.g., 
$\bm v^{\odot 2}:=(v_1^2,v_2^2,\ldots,v_N^2)^T$ and $\bm u \odot \bm v:=(u_{1}v_1,u_{2}v_2,\ldots,u_{N}v_N)^T.$
\subsection{First-order low regularity integrators}

Let us take  $\bm U_{0} = \bm u(0)$. For first-order schemes, we approximate $\bm u(t_m+s)$ by $e^{s\bA}\bm U_m$ in \eqref{exact1} and obtain:
for $m=0,1,\ldots,M-1$,
\begin{align}
\bm U_{m+1}&=e^{\Delta t \bA} \bm U_m + \int_0^{\Delta t} e^{(\Delta t -s)\bA} \bmf(e^{s\bA}\bm U_m)\,ds =e^{\Delta t \bA} \bm U_m + \int_0^{\Delta t} \bm g(s)\,ds,   \label{1stscheme}
\end{align}
where $\bm g(s)=e^{(\Delta t -s)\bA} \bmf(e^{s\bA}\bm U_m)$.   We then compute the integral on the right-hand side using one-point quadrature rules,  
which leads to the following schemes:
\begin{itemize}
    \item[i)] {\bf LRI1a} with $\bm g(s)$ on $[0,\Delta t]$ approximated by $\bm g(0)$ (as in \cite{LRIgeneralframework}):
\begin{align}\label{LRI1a} 
    \bm U_{m+1}=e^{\Delta t \bA}(\bm U_m+\Delta t\bmf(\bm U_m)).
\end{align}
\item[ii)] {\bf LRI1b} with $\bm g(s)$ on $[0,\Delta t]$ approximated by $\bm g(\Delta t)$: 
    \begin{align}\label{LRI1b}
   \bm U_{m+1}=e^{\Delta t\bA} \bm U_m + \Delta t \bmf(e^{\Delta t\bA} \bm U_m).
    \end{align}
\end{itemize}
We remark that the LRI1a scheme is identical to the first-order IFRK scheme~\cite{SSPRK,boundeA}, while the LRI1b scheme (first proposed in~\cite{LRINS}) is essentially different from the IFRK or ETD schemes. 

\subsection{Second-order low regularity integrator}

Let us approximate $\bm u(t_m+s)$ by $e^{s\bA}\bm U_m+s\bmf(\bm U_m)$  from \eqref{exact2}, and plug that into \eqref{exact1} to obtain: for $m=0,1,\ldots,M-1$,
$$\bm U_{m+1}= e^{\Delta t \bA}\bm U_m + \int_0^{\Delta t} e^{(\Delta t -s)\bA} \bmf(e^{s\bA}\bm U_m+s\bmf(\bm U_m))\,ds.$$
By replacing $\bmf(e^{s\bA}\bm U_m+s\bmf(\bm U_m))$ by $\bmf(e^{s\bA}\bm U_m)+s\frac{\partial \bmf}{\partial \bm u}(e^{s\bA}\bm U_m)\bmf(\bm U_m)$,  we arrive at
\begin{align*}
\bm U_{m+1}=e^{\Delta t\bA}\bm U_m + \int_0^{\Delta t}\bm g(s)\,ds + \int_0^{\Delta t}se^{(\Delta t -s)\bA}\frac{\partial \bmf}{\partial \bm u}(e^{s\bA}\bm U_m)\bmf(\bm U_m)\,ds.
\end{align*}
Finally,  the LRI2 scheme is defined by computing the integral $\int_0^{\Delta t}\bm g(s)\,ds$ using the trapezoid rule and  approximating $e^{(\Delta t -s)\bA}\frac{\partial \bmf}{\partial \bm u}(e^{s\bA}\bm U_m)\bmf(\bm U_m)$ by $e^{\Delta t \bA}\frac{\partial \bmf}{\partial \bm u}(\bm U_m)\bmf(\bm U_m)$ in the above equation:
\begin{align}\label{LRI2a}
\bm U_{m+1} = e^{\Delta t\bA} \bm U_m +\frac{\Delta t}{2}\left[e^{\Delta t\bA} \bmf(\bm U_m)+\bmf(e^{\Delta t\bA} \bm U_m)\right] +\frac{\Delta t^2}{2}e^{\Delta t\bA} \frac{\partial \bmf}{\partial \bm u}(\bm U_m) \bmf(\bm U_m). \vspace{-0.2cm}
\end{align}

\section{Properties of the fully discrete numerical solutions}\label{section_properties}

\subsection{Discrete maximum bound principle}\label{section_MBP}

We shall prove that the LRI1 and LRI2 schemes proposed in the previous section satisfy the MBP in the discrete sense.
We first state some preliminary result needed for the MBP proof.  It is easy to see from assumptions (A1) and (A2) that 
$\min\limits_{ |x|\le \beta} f'(x)<0$. Define $\omega_0:=-\frac{1}{\min\limits_{ |x|\le \beta} f'(x)}>0$. 

\begin{lemma}\label{lemma_ftongquat}
It holds that
$$|x+\omega f(x)|\le \beta \quad \text{ for all }\thickspace x\in [-\beta,\beta]\thickspace\text{ and } \thickspace\omega\in (0,\omega_0].$$
\end{lemma}
\begin{proof}
Let $f_0(x)=x+\omega f(x)$, we have 
$f_0'(x)=1+\omega f'(x)\ge 0$, for any $\omega\in (0,\omega_0]$ and $x\in [-\beta,\beta].$ Moreover, as $f(\beta)\le 0\le f(-\beta)$,  we deduce that
$$-\beta\le f_0(-\beta)\le f_0(x)\le f_0(\beta)\le \beta \quad \text{ for all }\thickspace x\in [-\beta,\beta], $$
which concludes the proof. 
\end{proof}
Both LRI1a and LRI1b schemes preserve the discrete MBP under certain time step constraint as stated below. 
\begin{theorem} \label{thrm:MBP_LRI1}
Suppose that $\|u_0\|_{L^{\infty}}\le \beta$, then the numerical solution $\{\bm U_m\}$  generated by the LRI1a scheme (cf.  Equation~\eqref{LRI1a}) or the  LRI1b scheme (cf.  Equation~\eqref{LRI1b}) is also pointwisely bounded in the absolute value by $\beta$, i.e.,
$$\|\bm U_m\|_{\infty}\le \beta, \quad m=0,1,\ldots,M,$$
provided that $0<\Delta t\le \omega_0.$
\end{theorem}
\begin{proof}
By induction, we only need to prove: if $\|\bm U_m\|_{\infty}\le \beta$, then it also holds that $\|\bm U_{m+1}\|_{\infty}\le~\beta$. 
For the LRI1a scheme,  by applying Lemmas \ref{bode} and \ref{lemma_ftongquat}, we deduce from~\eqref{LRI1a} that
$$\|\bm U_{m+1}\|_{\infty}\le \|e^{\Delta t \bA}\|_{\infty}\|\bm U_m+ \Delta t \bmf(\bm U_m)\|_{\infty}\le \beta \quad \text{ for all }\thickspace \Delta t \in (0,\omega_0].$$ 
For the LRI1b scheme,  as $\|e^{\Delta t \bA}\bm U_m\|_{\infty}\le \beta$ (Lemma \ref{bode}),  together with Lemma \ref{lemma_ftongquat} we obtain
$$\|\bm U_{m+1}\|_{\infty}=\|e^{\Delta t\bA} \bm U_m + \Delta t \bmf(e^{\Delta t\bA} \bm U_m)\|_{\infty}\le \beta \quad \text{ for all }\thickspace\Delta t\in (0,\omega_0].$$
The proof is thus completed.
\end{proof}
Unlike the LRI1 schemes, the LRI2 scheme requires a stronger condition on $f$ to obtain the discrete MBP, namely: \vspace{5pt}\\ 
\noindent (A3) {\it $f$ is twice continuously differentiable and nonconstant on $[-\beta,\beta].$}  \vspace{5pt}\\
The MBP preservation of the LRI2 scheme is guaranteed by the following theorem.  Define
$$\omega_1:=-\min_{|x|\le\beta}[f''(x)f(x)].$$
\begin{theorem} \label{thrm:MBP_LRI2}
Suppose that $f$ satisfies  (A1) and (A3).   If $\|u_0\|_{L^{\infty}}\le \beta$,  then the numerical solution $\{\bm U_m\}$  generated by the LRI2 scheme (cf.  Equation~\eqref{LRI2a}) is also pointwisely bounded in the absolute value by $\beta$, i.e.,
$$\|\bm U_m\|_{\infty}\le \beta, \quad\thickspace m=0,1,\ldots,M,$$
provided that $0<\Delta t \le \delta_0\omega_0$, where $\delta_0$ is given by
$$\delta_0=\left \{ \begin{array}{cl}
    \min\{1,\delta\}, & \text{if } \omega_1>0,\\
    1, & \text{if } \omega_1\le 0, 
    \end{array}\right.$$
with $\delta=\frac{-1+\sqrt{1+7\omega_0^2\omega_1}}{2\omega_0^2\omega_1}$ if $\omega_1>0$.
\end{theorem}

\begin{proof}
By induction,  we again only need to show that if $\|\bm U_m\|_{\infty}\le \beta$ then $\|\bm U_{m+1}\|_{\infty}\le \beta.$ Apply Lemma \ref{bode}, for any $\delta_1\in (0,2)$ we have
\begin{align*}
\|\bm U_{m+1}\|_{\infty}&\le \left\|\frac{\delta_1}{2} e^{\Delta t \bA}\bm U_m +\frac{\Delta t}{2}\bmf(e^{\Delta t \bA}\bm U_m)\right\|_{\infty} +\left\|\frac{2-\delta_1}{2}\bm e^{\Delta t \bA}U_m +\frac{\Delta t}{2}e^{\Delta t \bA}\bmf(\bm U_m)+\frac{\Delta t^2}{2}e^{\Delta t \bA}\frac{\partial \bmf}{\partial \bm u}(\bm U_m)\bmf(\bm U_m)\right\|_{\infty}    \\
&\le\frac{\delta_1}{2}\left\|e^{\Delta t \bA}\bm U_m+\frac{\Delta t}{\delta_1}\bmf(e^{\Delta t \bA}\bm U_m)\right\|_{\infty}+\frac 12 \left\|(2-\delta_1)\bm U_m +\Delta t\bmf(\bm U_m)+\Delta t^2 \frac{\partial \bmf}{\partial \bm u}(\bm U_m)\bmf(\bm U_m)\right\|_{\infty}.
\end{align*}
\Rv{Notice that $\|e^{\Delta t \bA}\bm U_m\|_{\infty}\le\beta$.  If $\frac{\Delta t}{\delta_1}\in (0,\omega_0]$ (i.e., $\Delta t\in (0,\delta_1\omega_0]$), then we observe  by Lemma \ref{lemma_ftongquat} that
$$\left\|e^{\Delta t \bA}\bm U_m+\frac{\Delta t}{\delta_1}\bmf(e^{\Delta t \bA}\bm U_m)\right\|_{\infty}\le\beta.$$}
Define $K(x)=(2-\delta_1)x + \Delta t f(x) + \Delta t^2 f'(x)f(x)$ for $x\in [-\beta,\beta].$ \Rv{Rewrite $K(x)$ as follows:
$$K(x)=(2-\delta_1)x + (1+\Delta tf'(x))\Delta t f(x).$$
If $\Delta t\in (0,\omega_0]$, one can apply the property of $f_0'(x)$ in the proof of Lemma \ref{lemma_ftongquat} to obtain}
\begin{align}\label{K function}
 -(2-\delta_1)\beta\le K(-\beta),  \; \text{ and } \; K(\beta)\le (2-\delta_1)\beta.   
\end{align}
On the other hand,
\begin{align*}
 K'(x)&=2-\delta_1 + \Delta t f'(x) + \Delta t^2 [f'(x)]^2 + \Delta t^2 f''(x)f(x)  \\
&=\left(\frac 12 +\Delta t f'(x)\right)^2 + \frac 74-\delta_1+\Delta t^2 f''(x)f(x)\\
 &\ge \frac 74 -\delta_1 - \Delta t^2 \omega_1 \quad \text{ for all }\thickspace x\in [-\beta,\beta].
\end{align*} 
\Rv{If $\delta_1< \frac 74$, and $\Delta t\le \sqrt{\frac{1}{\omega_1}\left(\frac 74-\delta_1\right)}$ in case of $\omega_1>0$, then we have
 $$K'(x)\ge 0\quad \text{ for all }\thickspace x\in [-\beta,\beta].$$}
Therefore, $K$ is non-decreasing on $[-\beta,\beta]$. This, combines with \eqref{K function}, gives us
$$|K(x)|\le (2-\delta_1)\beta \quad \text{ for all }\thickspace x\in [-\beta,\beta].$$
Thus,
\begin{align}\label{MBP2_con}
\|\bm U_{m+1}\|_{\infty}\le \frac{\delta_1}{2}\beta+ \frac{2-\delta_1}{2}\beta=\beta,
\end{align}
provided that the following conditions are all satisfied
\begin{align*}
\delta_1&\in \left(0,\frac 74\right),\quad \Delta t\in (0,\delta_1\omega_0],\quad \Delta t \in (0,\omega_0],\\
\Delta t&\le \sqrt{\frac{1}{\omega_1}\left(\frac 74-\delta_1\right)}\;\text{ if } \;\omega_1>0.
\end{align*}
\Rv{If $\omega_1\le 0$, we can choose $\delta_1=1$ and the condition on the time step size is $\Delta t\in (0,\omega_0]$. If $\omega_1>0$ then $\Delta t$ has to satisfy the constraint below
$$0<\Delta t\le \min\left\{\omega_0,\delta_1\omega_0,\sqrt{\frac{1}{\omega_1}\left(\frac 74-\delta_1\right)}\right\},\quad \delta_1\in \left(0,\frac 74\right),$$
or equivalently,
$$0<\Delta t\le \omega_0\min\left\{1,\delta_1,\sqrt{\frac{1}{\omega_0^2\omega_1}\left(\frac 74-\delta_1\right)}\right\}.$$
The best choice of the constant $\delta_1$ in this case is $\delta_1=\delta$ because of
the the following relations:
$$\delta=\sqrt{\frac{1}{\omega_0^2\omega_1}\left(\frac 74-\delta\right)},\:\:\: \delta\in \left(0,\frac 74\right).$$
Hence, we require $0<\Delta t\le \omega_0\min\{1,\delta\}$ when $\omega_1>0$.  Consequently, \eqref{MBP2_con} is true for all $\Delta t\in (0,\delta_0\omega_0]$, which completes the proof of  Theorem~\ref{thrm:MBP_LRI2}}.
\end{proof}
\begin{remark} \label{rmk:MBPtimestep}
If the nonlinear function $f$ satisfies additionally $f(-\beta)=f(\beta)=0$ for some $\beta>0$, then \eqref{K function} is obviously true without the requirement $0<\Delta t\le \omega_0$. Also in this case, it is not difficult to prove that $\omega_1>0$, meaning that $\delta$ exists. Thus, the range of $\Delta t$ can be enlarged, particularly, $0<\Delta t \le \delta \omega_0.$ This is the case for both      double-well potential and Flory-Huggins potential functions \cite{Review}.  
The double-well potential function $F(u)=\frac 14(u^2-1)^2$ gives the nonlinear term $f(u)=u-u^3$, then  $f(-1)=f(1)=0$, i.e., $\beta=1$. By simple calculations, we obtain $$\omega_0=\frac 12, \quad \omega_1=\frac 32, \quad \delta=\frac{-4+\sqrt{58}}{3}\approx \frac 65.$$ Therefore, the MBP is preserved when $0<\Delta t \le \frac 12$ (for the LRI1 schemes) and $0<\Delta t\le \frac 35$ (for the LRI2 scheme).
The Flory-Huggins potential function is given by
$F(u)=\frac{\theta}{2}[(1+u)\ln(1+u)+(1-u)\ln(1-u)]-\frac{\theta_c}{2}u^2,$
where $0<\theta<\theta_c$. Then $f(u)=-F'(u)$, namely,
$f(u)=\frac{\theta}{2}\ln\frac{1-u}{1+u}+\theta_cu.$
Consider a particular case~\cite{boundeA} where $\theta=0.8$ and $\theta_c=1.6$, then $\beta\approx 0.9575$ satisfying $f(-\beta)=f(\beta)=0$. We can easily verify that
\begin{align*}
 \omega_0=\frac{1-\beta^2}{\theta-\theta_c(1-\beta^2)}\approx 0.1247,\quad
 \omega_1\approx - f(0.932)f''(0.932)\approx 13.1739,\quad \delta\approx 1.367.
\end{align*}
Thus,  the conditions on $\Delta t$ to preserve the MBP are $0<\Delta t\le 0.1247$ for the LRI1  schemes and $0<\Delta t\le 0.1705$ for the LRI2 scheme. We shall verify these time step constraints in the numerical experiments (cf. Section~\ref{section_numerical}). 
\end{remark}

\subsection{Discrete energy stability}\label{section_energy}

The discrete energy $E_h$ is defined as
\begin{align} \label{discrete_energy}
E_h(\bm U)=\sum_{i=1}^N F(U^i) - \frac{\varepsilon^2}{2}\bm U^T\bD_h \bm U \quad \text{ for all } \thickspace\bm U=(U^1, U^2,\ldots, U^N)^T\in\mathbb{R}^N,
\end{align}
 where $F$ is the potential function.  In the following, we show that the first- and second-order LRI schemes are conditionally energy stable in the sense that the discrete energy is bounded at all times.  We first state some preliminary results needed for the main theorem.  Let us define $F_0=\max\limits_{|x|\le \beta}|f(x)|$ and $F_1= \max\limits_{|x|\le \beta}|f'(x)|$.

\begin{lemma}\label{Edifference}
If $\|\bm U_m\|_{\infty}\le \beta$  for all $m=0,1, \hdots, M$, then there exists $\bW$ in~$\mathbb{R}^N$, $\|\bW\|_{\infty}\le F_0$ such that
$$E_h(\bm U_{m+1})-E_h(\bm U_m)\le (\bm U_{m+1}-\bm U_m)^T(\bW- \bA\bm U_{m+1}),$$
for $m=0,1, \hdots, M-1$. 
\end{lemma}

\begin{proof} We first have the following identity from \eqref{discrete_energy}:
\begin{equation}\label{energy1}
 E_h(\bm U_{m+1})-E_h(\bm U_m)=\sum_{i=1}^N [F(U^i_{m+1})-F(U^i_{m})] -\frac{\varepsilon^2}{2}(\bm U_{m+1}^T\bD_h\bm U_{m+1}-\bm U_{m}^T\bD_h\bm U_{m}).   
\end{equation}
Since $\|\bm U_m\|_{\infty}\le \beta$ for $m=0,1,\ldots,M$, by the mean value theorem, there exist $\gamma_1,\gamma_2,\ldots,\gamma_N$ such that $|\gamma_i|\le \beta$ for $i=1,2,\ldots,N$ and
\begin{align*}
F(U^i_{m+1})-F(U^i_{m})=(U^i_{m+1}-U^i_{m})F'(\gamma_i)=-(U^i_{m+1}-U^i_{m})f(\gamma_i), \quad i=1,2,\ldots,N.
\end{align*}
Let $\bW:=-[f(\gamma_1),f(\gamma_2),\ldots,f(\gamma_N)]^T$, we can see that $\|\bW\|_{\infty}\le F_0$ and
\begin{align}\label{energy2}
\sum_{i=1}^N &[F(U^i_{m+1})-F(U^i_{m})]=(\bm U_{m+1}-\bm U_m)^T\bW.
\end{align}
On the other hand, $\bD_h$ is negative semidefinite and symmetric, so
\begin{align*}
0\ge (\bm U_{m+1}-\bm U_m)^T \bD_h (\bm U_{m+1}-\bm U_m)&= \bm U_{m+1}^T \bD_h \bm U_{m+1} + \bm U_m^T \bD_h \bm U_m - \bm U_{m+1}^T \bD_h \bm U_m - \bm U_m^T \bD_h \bm U_{m+1}\\
& = -\bm U_{m+1}^T \bD_h \bm U_{m+1}+\bm U_m^T \bD_h \bm U_m +2(\bm U_{m+1}-\bm U_m)^T \bD_h \bm U_{m+1}.
\end{align*}
Therefore,
\begin{align}\label{energy3}
  \bm U_{m+1}^T \bD_h \bm U_{m+1}-\bm U_m^T \bD_h \bm U_m \ge 2(\bm U_{m+1}-\bm U_m)^T \bD_h \bm U_{m+1}.  
\end{align}
From \eqref{energy1}, \eqref{energy2} and \eqref{energy3}, we get
$$E_h(\bm U_{m+1})-E_h(\bm U_m) \le (\bm U_{m+1}-\bm U_m)^T(\bW- \bA\bm U_{m+1}),$$
which leads to the conclusion.
\end{proof}

\begin{lemma}\label{lemma_consecutive_U}$ $
\begin{enumerate}
\item[(i)] If $0<\Delta t\le \omega_0$ and $f$ satisfies (A1) and (A2), then the numerical solution   $\{\bm U_m\}$ generated by the LRI1a scheme~\eqref{LRI1a} or the LRI1b scheme~\eqref{LRI1b}  satisfies 
    $$\|\bm U_{m+1}-\bm U_m\|_{\infty}\le (\beta \|\bA\|_{\infty}+F_0)\Delta t.$$
\item[(ii)] If $0<\Delta t\le \delta_0\omega_0$ and $f$ satisfies (A1) and (A3),  then the numerical solution   $\{\bm U_m\}$ generated by the LRI2 scheme~\eqref{LRI2a} satisfies
    $$\|\bm U_{m+1}-\bm U_m\|_{\infty}\le (\beta \|\bA\|_{\infty}+F_0+\delta_0\omega_0 F_0F_1)\Delta t.$$
\end{enumerate}
\end{lemma}

\begin{proof} By Theorem \ref{thrm:MBP_LRI1} or \ref{thrm:MBP_LRI2}, we know that the numerical solution $\{\bm U_m\}$ generated by the schemes~\eqref{LRI1a}, \eqref{LRI1b} or \eqref{LRI2a} preserves the MBP, meaning that $\|\bm U_m\|_{\infty}\le \beta$ \text{for} $m=0,1,\ldots,M.$

{\it (i)} Consider the LRI1a scheme~\eqref{LRI1a}.  Define $\bq_1(t)=e^{t\bA}\bm U_m+te^{\Delta t\bA} \bmf(\bm U_m)$, then $\bq_1(0)=\bm U_m$, $\bq_1(\Delta t)=\bm U_{m+1}$, and 
$$\bq_1'(t)=\bA e^{t\bA}\bm U_m+e^{\Delta t \bA}\bmf(\bm U_m).$$
By Lemma \ref{bode}, we can easily see that
$$\|\bq_1'(t)\|_{\infty}\le \beta \|\bA\|_{\infty}+F_0 \quad \text{ for all }\thickspace t> 0.$$
Apply the mean value theorem, we get
$$\|\bm U_{m+1}-\bm U_m\|_{\infty}=\|\bq_1(\Delta t)-\bq_1(0)\|_{\infty}\le (\beta\|\bA\|_{\infty}+F_0)\Delta t.$$
Similarly for the LRI1b scheme~\eqref{LRI1b}, define $\bq_2(t)=e^{t\bA}\bm U_m + t\bmf(e^{\Delta t\bA}\bm U_m)$, then $\bq_2(0)=\bm U_m$ and $\bq_2(\Delta t)=\bm U_{m+1}$. Moreover,
$$\|\bq_2'(t)\|_{\infty}=\|\bA e^{t\bA}\bm U_m + \bmf(e^{\Delta t \bA}\bm U_m)\|_{\infty}\le \beta \|\bA\|_{\infty}+F_0 \quad \text{ for all }\thickspace t>0.$$
The result of (i) then follows by the mean value theorem.

{\it (ii)} For the LRI2 scheme~\eqref{LRI2a},  define $\bq_3(t)=\frac{\bq_1(t)+\bq_2(t)}{2}+\frac{t^2}{2}e^{\Delta t \bA}\frac{\partial\bmf}{\partial\bm u}(\bm U_m)\bmf(\bm U_m)$, then $\bq_3(0)=\bm U_m$, $\bq_3(\Delta t)=\bm U_{m+1}$ and
$$\bq_3'(t)=\frac{\bq_1'(t)+\bq_2'(t)}{2}+te^{\Delta t \bA}\frac{\partial\bmf}{\partial\bm u}(\bm U_m)\bmf(\bm U_m).$$
Thus we have
\begin{align*}
\|\bq_3'(t)\|_{\infty}&\le \frac{\|\bq_1'(t)\|_{\infty}+\|\bq_2'(t)\|_{\infty}}{2}+\delta_0\omega_0F_0F_1 \le \beta\|\bA\|_{\infty}+F_0+\delta_0\omega_0F_0F_1 \quad \text{ for all }\thickspace t\in [0,\Delta t].
\end{align*}
Then the result of (ii)  follows by the mean value theorem.
\end{proof}

\begin{theorem}
Suppose $0<\Delta t\le \omega_0$ (resp. $0<\Delta t\le \delta_0\omega_0$) and $f$ satisfies (A1) and (A2) (resp. (A1) and (A3)), then the numerical solution $\{\bm U_m\}$ generated by the LRI1a scheme~\eqref{LRI1a} or the LRI1b scheme~\eqref{LRI1b} (resp.  the LRI2 scheme~\eqref{LRI2a}) satisfies
\begin{equation}
E_h(\bm U_m)\le E_h(\bm U_0)+C, \quad m=0,1,\ldots,M,
\end{equation}
where the constant $C>0$ is independent of $\Delta t.$
\end{theorem}

\begin{proof}
According to Lemma \ref{lemma_consecutive_U}, there exists $c>0$ independence of $\Delta t$ such that 
$$\|\bm U_{m+1}-\bm U_m\|_{\infty}\le c\Delta t.$$
This, together with Lemma \ref{Edifference}, gives us
\begin{align*}
E_h(\bm U_{m+1})-E_h(\bm U_m)&\le \|(\bm U_{m+1}-\bm U_m)^T\|_{\infty}(\|\bW\|_{\infty}+\|\bA\bm U_{m+1}\|_{\infty})\le cN\Delta t(F_0+\beta\|\bA\|_{\infty}).
\end{align*}
Therefore,
$$E_h(\bm U_m)\le E_h(\bm U_0)+cNT(F_0+\beta\|\bA\|_{\infty})=:E_h(\bm U_0)+C,$$
where $C=cNT(F_0+\beta\|\bA\|_{\infty})>0$ independence of $\Delta t.$
\end{proof}

\section{Temporal error estimates}\label{section_error_estimates}

In this section, we always assume $\|u_0\|_{L^{\infty}}\le \beta$ so that 
$$\bm u\in \mathcal{C}([0,T];\mathbb{R}^N) \:\text{ and }\:\|\bm u(t)\|_{\infty}\le \beta \;\text{ for all }\; t\in [0,T], $$ by Theorem \ref{MBP_exact_solution} and the discrete MBP of the numerical solution $\{\bm U_m\}$ also holds by Theorem~\ref{thrm:MBP_LRI1} or Theorem~\ref{thrm:MBP_LRI2}. By using Duhamel's formula as the key idea, we will show that these requirements on $\bm u$ (the continuity of $\bm u(t)$ and its MBP) are enough to obtain the temporal error estimates (under a fixed spatial mesh size) for the proposed LRI schemes.  Let $\bm e_m:=\bm U_m-\bm u(t_m)$ be the error at $t_m=m\Delta t\le T$. Let us define
$F_2=\max\limits_{|x|\le \beta} |f''(x)|$, $F_3=\max\limits_{|x|\le \beta}|(f'(x)f(x))'|$,
$\widetilde{F_1}=\max\limits_{|x|\le \beta + \delta_0\omega_0F_0}|f'(x)|$, and $\widetilde{F_2}=\max\limits_{|x|\le \beta + \delta_0 \omega_0F_0}|f''(x)|$.
The following result will be useful for the error analysis of the LRI schemes. 
\begin{lemma}\label{lemma_phi_psi}
For each $0\le s \le \Delta t$, define \vspace{-0.2cm}
$$\bm\lambda(s)=e^{(\Delta t -s)\bA}\bmf(e^{s\bA}\bm u(t_m))\text{ and } \bm\psi(s)=e^{(\Delta t -s)\bA}\frac{\partial \bmf}{\partial \bm u}(e^{s\bA}\bm u(t_m))\bmf(\bm u(t_m)). $$
Then \vspace{-0.1cm}
\begin{itemize}
    \item[(i)] $\|\bm\lambda'(s)\|_{\infty}\le (F_0+\beta F_1)\|\bA\|_{\infty}$ and $\|\bm\lambda''(s)\|_{\infty}\le (F_0+3\beta F_1 + \beta^2 F_2) \|\bA\|_{\infty}^2.$\vspace{0.2cm}
    \item[(ii)] $\|\bm\psi'(s)\|_{\infty}\le F_0(F_1+\beta F_2)\|\bA\|_{\infty}.$
\end{itemize}
\end{lemma}

\begin{proof}
We have \vspace{-0.2cm}
$$\bm\lambda'(s)=e^{(\Delta t -s)\bA}\left[-\bA\bmf(e^{s\bA}\bm u(t_m))+\frac{\partial \bmf}{\partial \bm u}(e^{s\bA}\bm u(t_m))(\bA e^{s\bA}\bm u(t_m))\right].$$
Since $\|e^{s\bA}\bm u(t_m)\|_{\infty}\le \beta$,  this implies
$$\|\bm\lambda'(s)\|_{\infty}\le F_0\|\bA\|_{\infty}+\beta F_1\|\bA\|_{\infty}=(F_0+\beta F_1)\|\bA\|_{\infty}.$$
On the other hand, \vspace{-0.2cm}
\begin{align*}
 \bm\lambda''(s)=&\;e^{(\Delta t -s)\bA}\bigg[\bA^2\bmf(e^{s\bA}\bm u(t_m))-2\bA\frac{\partial \bmf}{\partial \bm u}(e^{s\bA}\bm u(t_m))(\bA e^{s\bA}\bm u(t_m))\\
&+\frac{\partial^2 \bmf}{\partial \bm u^2}(e^{s\bA}\bm u(t_m))(\bA e^{s\bA}\bm u(t_m))^{\odot 2} + \frac{\partial \bmf}{\partial \bm u}(e^{s\bA}\bm u(t_m))(\bA^2e^{s\bA}\bm u(t_m))\bigg].   
\end{align*}
Thus we get
\begin{align*}
\|\bm\lambda''(s)\|_{\infty}&\le F_0\|\bA\|_{\infty}^2+2\beta F_1\|\bA\|_{\infty}^2+\beta^2F_2\|\bA\|_{\infty}^2+\beta F_1\|\bA\|_{\infty}^2=(F_0+3\beta F_1+\beta^2F_2)\|\bA\|_{\infty}^2.
\end{align*}
Similarly for $\bm\psi$, we have
\begin{align*}
\bm\psi'(s)=e^{(\Delta t -s)\bA}&\bigg[-\bA\frac{\partial \bmf}{\partial \bm u}(e^{s\bA}\bm u(t_m))\bmf(\bm u(t_m))+\frac{\partial^2 \bmf}{\partial \bm u^2}(e^{s\bA}\bm u(t_m))(\bA e^{s\bA}\bm u(t_m))\odot\bmf(\bm u(t_m))\bigg].
\end{align*}
Therefore,
$$\|\bm\psi'(s)\|_{\infty}\le F_0(F_1\|\bA\|_{\infty}+\beta F_2\|\bA\|_{\infty})=F_0(F_1+\beta F_2)\|\bA\|_{\infty}.$$
The proof is then completed.
\end{proof}

Next, we prove the convergence of the numerical solutions by the LRI1a,  LRI1b and LRI2 schemes to the exact solution of the space-discrete problem \eqref{semi-discrete} as $\Delta t$ tends to zero, and their corresponding error estimates. 
\begin{theorem}\label{error_estimate_1a}  
Given any fixed terminal time $T>0$ and spatial mesh size $h>0$. Assume that $\bm u\in \mathcal{C}([0,T];\mathbb{R}^N)$ is the exact solution of \eqref{semi-discrete} and $\{\bm U_m\}$ is the numerical solution generated by the LRI1a scheme (cf. Equation~\eqref{LRI1a}). Then, there exists a constant $C=C(\beta, F_0,F_1,\|\bA\|_{\infty})>0$ independent of $\Delta t$ such that the following estimate is true for all $0<\Delta t \le \omega_0$:
$$\|\bm U_m-\bm u(t_m)\|_{\infty} \le C(e^{F_1t_m}-1)\Delta t, \quad m=0,1,\ldots,M.$$
\end{theorem}

\begin{proof}
For the LRI1a scheme \eqref{LRI1a}, we have 
\begin{align}\label{trucation1}
 \bm u(t_{m+1})=e^{\Delta t \bA}\left(\bm u(t_m)+\Delta t \bmf(\bm u(t_m))\right) + \bR_1(t_m),   
\end{align}
where $\bR_1(t_m)$ is the corresponding truncation error.
This together with \eqref{LRI1a} gives us
$$\bm e_{m+1}=e^{\Delta t \bA}\left[\bm e_m+\Delta t \left(\bmf(\bm U_m)-\bmf(\bm u(t_m))\right)\right] - \bR_1(t_m).$$
Since $\|e^{\Delta t \bA}\|_{\infty} \le 1$, we have
\begin{align}\label{errorR1.1}
\|\bm e_{m+1}\|_{\infty} \le \|\bm e_m\|_{\infty} + \Delta t \|\bmf(\bm U_m)-\bmf(\bm u(t_m))\|_{\infty} + \|\bR_1(t_m)\|_{\infty}.    
\end{align}
Observe that both the exact and numerical solutions of \eqref{semi-discrete} satisfy the MBP when $0<\Delta t \le \omega_0$, meaning that $\|\bm U_m\|_{\infty}\le \beta, \|\bm u(t_m)\|_{\infty} \le \beta$.  By using the Lipschitz continuity of $\bmf$, we obtain
\begin{align}\label{errorR1.2}
 \|\bmf(\bm U_m)-\bmf(\bm u(t_m))\|_{\infty} \le F_1\|\bm U_m-\bm u(t_m)\|_{\infty}=F_1\|\bm e_m\|_{\infty}.
\end{align}
On the other hand,  by comparing \eqref{exact1} and \eqref{trucation1}, we can rewrite $\bR_1(t_m)$ as follows:
\begin{align*}
  \bR_1(t_m)&=\int_0^{\Delta t} e^{(\Delta t -s)\bA}\bmf(\bm u(t_m+s))\,ds -\Delta t e^{\Delta t \bA} \bmf(\bm u(t_m)) = \int_0^{\Delta t} \left[e^{(\Delta t-s)\bA}\bmf(\bm u(t_m+s))-e^{\Delta t \bA}\bmf(\bm u(t_m))\right]ds.
\end{align*}
Therefore,
\begin{align*}
\|\bR_1(t_m)\|_{\infty} &\le \int_0^{\Delta t} \|e^{\Delta t \bA}\bmf(\bm u(t_m))- e^{(\Delta t-s)\bA}\bmf(\bm u(t_m+s))\|_{\infty}\,ds\\
&\le \int_0^{\Delta t} \left (\|e^{\Delta t \bA}\bmf(\bm u(t_m))-e^{(\Delta t-s)\bA}\bmf(e^{s\bA}\bm u(t_m))\|_{\infty} + \|e^{(\Delta t-s)\bA}\bmf(e^{s\bA}\bm u(t_m))- e^{(\Delta t-s)\bA}\bmf(\bm u(t_m+s))\|_{\infty}\right )\,ds\\
    &\le \int_0^{\Delta t} \left (\|e^{\Delta t \bA}\bmf(\bm u(t_m))-e^{(\Delta t-s)\bA}\bmf(e^{s\bA}\bm u(t_m))\|_{\infty} + \|\bmf(e^{s\bA}\bm u(t_m))- \bmf(\bm u(t_m+s))\|_{\infty}\right )\,ds\\
    &=: \int_0^{\Delta t} (Q_1+Q_2)\,ds,
\end{align*}
where 
\begin{align*}
Q_1=\|e^{\Delta t \bA}\bmf(\bm u(t_m))-e^{(\Delta t-s)\bA}\bmf(e^{s\bA}\bm u(t_m))\|_{\infty},\quad Q_2=\|\bmf(e^{s\bA}\bm u(t_m))- \bmf(\bm u(t_m+s))\|_{\infty}.
\end{align*}
Using Lemma \ref{lemma_phi_psi} and the mean value theorem, we deduce that
\begin{equation}\label{Q1est}
Q_1=\|\bm\lambda(0)-\bm\lambda(s)\|_{\infty}\le s(F_0+\beta F_1)\|\bA\|_{\infty}.
\end{equation}
To estimate $Q_2$,  we apply the formula \eqref{exact2} to obtain
\begin{align}\label{Q2est}
 Q_2&\le F_1\|e^{s\bA}\bm u(t_m)-\bm u(t_m+s)\|_{\infty}=F_1\left\|\int_0^se^{(s-\sigma)\bA}\bmf(\bm u(t_m+\sigma))\,d\sigma\right\|_{\infty} \nonumber  \\
 & \le F_1\int_0^s \|e^{(s-\sigma)\bA}\|_{\infty}\|\bmf(\bm u(t_m+\sigma))\|_{\infty}d\sigma \le F_0F_1s.
\end{align}
Thus we obtain from \eqref{Q1est} and \eqref{Q2est} that
\begin{align}
 \|\bR_1(t_m)\|_{\infty}&\le \int_0^{\Delta t}(Q_1+Q_2)\,ds \le \int_0^{\Delta t} [(F_0+\beta F_1)\|\bA\|_{\infty}+F_0F_1]s\,ds  \notag\\
 & =\frac 12[(F_0+\beta F_1)\|\bA\|_{\infty}+F_0F_1]\Delta t^2  =:  c_0\Delta t^2. \label{errorR1.3}
\end{align}
By \eqref{errorR1.1}, \eqref{errorR1.2} and \eqref{errorR1.3},  the following estimate  holds:
\begin{equation}
\|\bm e_{m+1}\|_{\infty} \le (1+F_1\Delta t) \|\bm e_m\|_{\infty} + c_0\Delta t^2.
\end{equation}
This implies
\begin{align*}
\|\bm e_m\|_{\infty} + \frac{ c_0 \Delta t}{F_1}&\le (1+F_1\Delta t)\left(\|\bm e_{m-1}\|_{\infty} +\frac{c_0\Delta t}{F_1}\right) \le (1+F_1\Delta t)^m \left(\|\bm e_{0}\|_{\infty} +\frac{ c_0\Delta t}{F_1}\right).
\end{align*}
Note that $\bm e_0=0$,  thus
\begin{align*}
\|\bm e_m\|_{\infty} &\le \frac{ c_0\Delta t}{F_1}\left[(1+F_1\Delta t)^m-1\right]\le \frac{ c_0\Delta t}{F_1}(e^{F_1t_m}-1)=: C(e^{F_1t_m}-1)\Delta t,
\end{align*}
where $C=\frac{ c_0}{F_1}=\frac 12\left[\left(\frac{F_0}{F_1}+\beta\right)\|\bA\|_{\infty}+F_0\right]>0.$
\end{proof}

\begin{theorem}\label{error_estimate_1b}
Given any fixed terminal time $T>0$ and spatial mesh size $h>0$. Assume that $\bm u\in \mathcal{C}([0,T];\mathbb{R}^N)$ is the exact solution of \eqref{semi-discrete} and $\{\bm U_m\}$ is the numerical solution generated by the LRI1b scheme (cf.  Equation~\eqref{LRI1b}). Then, there exists a constant $C=C(\beta, F_0,F_1,\|\bA\|_{\infty})>0$ independent of $\Delta t$ such that the following estimate is true for all $0<\Delta t \le \omega_0$:
$$\|\bm U_m-\bm u(t_m)\|_{\infty} \le C(e^{F_1t_m}-1)\Delta t, \quad m=0,1,\ldots,M.$$
\end{theorem}

\begin{proof}
For the LRI1b scheme \eqref{LRI1b}, we have  
\begin{align}\label{truncation2}
 \bm u(t_{m+1})=e^{\Delta t \bA} \bm u(t_m) + \Delta t \bmf\left(e^{\Delta t \bA}\bm u(t_m)\right) + \bR_2(t_m),   
\end{align}
where $\bR_2(t_m)$ is the corresponding truncation error. 
This together with \eqref{LRI1b} gives us
$$\bm e_{m+1}=e^{\Delta t \bA}\bm e_m + \Delta t \left[\bmf\left(e^{\Delta t \bA}\bm U_m\right)-\bmf\left(e^{\Delta t \bA}\bm u(t_m)\right)\right] - \bR_2(t_m).$$
Since $\|e^{\Delta t \bA}\|_{\infty} \le 1$, we have
\begin{align}\label{error2.1}
 \|\bm e_{m+1}\|_{\infty} \le \|\bm e_m\|_{\infty} + \Delta t \left\|\bmf\left(e^{\Delta t \bA}\bm U_m\right)-\bmf\left(e^{\Delta t \bA}\bm u(t_m)\right)\right\|_{\infty} + \|\bR_2(t_m)\|_{\infty}.   
\end{align}
Notice that $\|e^{\Delta t \bA}\bm u(t_m)\|_{\infty}\le \beta$ and $\|e^{\Delta t \bA}\bm U_m\|_{\infty}\le \beta$ under the condition $0<\Delta t \le \omega_0$. Using the Lipschitz continuity of $\bmf$, we obtain
\begin{align}\label{error2.2}
 \left\|\bmf\left(e^{\Delta t \bA}\bm U_m\right)-\bmf\left(e^{\Delta t \bA}\bm u(t_m)\right)\right\|_{\infty} \le F_1 \left\|e^{\Delta t \bA}\bm U_m- e^{\Delta t \bA}\bm u(t_m)\right\|_{\infty}\le F_1\|\bm e_m\|_{\infty}.   
\end{align}
Compare \eqref{truncation2} and \eqref{exact1}, we can rewrite $\bR_2(t_m)$ as follows:
\begin{align*}
 \bR_2(t_m)&=\int_0^{\Delta t} e^{(\Delta t -s)\bA}\bmf\left(\bm u(t_m+s)\right)\,ds -\Delta t \bmf(e^{\Delta t \bA}\bm u(t_m))=\int_0^{\Delta t} \left [e^{(\Delta t -s)\bA}\bmf\left(\bm u(t_m+s)\right)-\bmf(e^{\Delta t \bA}\bm u(t_m))\right ]\,ds.
\end{align*}
Therefore,
\begin{align*}
\|\bR_2(t_m)\|_{\infty} &\le \int_0^{\Delta t} \left (\|\bmf(e^{\Delta t \bA}\bm u(t_m))-e^{(\Delta t -s)\bA}\bmf(e^{s\bA}\bm u(t_m))\|_{\infty} + \|e^{(\Delta t -s)\bA}\bmf(e^{s\bA}\bm u(t_m))-e^{(\Delta t -s)\bA}\bmf\left(\bm u(t_m+s)\right)\|_{\infty}\right )\, ds\\
&\le \int_0^{\Delta t} \left ( \|\bmf(e^{\Delta t \bA}\bm u(t_m))-e^{(\Delta t -s)\bA}\bmf(e^{s\bA}\bm u(t_m))\|_{\infty}+\|\bmf(e^{s\bA}\bm u(t_m))-\bmf\left(\bm u(t_m+s)\right)\|_{\infty}\right ) \,ds\\
&=: \int_0^{\Delta t} (Q_0+Q_2) ds,
\end{align*}
where 
$$Q_0=\|\bmf(e^{\Delta t \bA}\bm u(t_m))-e^{(\Delta t -s)\bA}\bmf(e^{s\bA}\bm u(t_m))\|_{\infty}, $$
and $Q_2$ is the same term as appeared in the proof of the previous theorem.
Using Lemma \ref{lemma_phi_psi} and the mean value theorem, we deduce that
\begin{equation}\label{Q0est}
Q_0=\|\bm\lambda(\Delta t)-\bm\lambda(s)\|_{\infty}\le (\Delta t -s)(F_0+\beta F_1)\|\bA\|_{\infty}.
\end{equation}
Thus we obtain from \eqref{Q0est} and \eqref{Q2est} that
\begin{align}
 \|\bR_2(t_m)\|_{\infty}&\le\int_0^{\Delta t}(Q_0+Q_2)\,ds\le \int_0^{\Delta t}[(\Delta t -s)(F_0+\beta F_1)\|\bA\|_{\infty}+F_0F_1s]\,ds \notag\\
 &=\frac 12[(F_0+\beta F_1)\|\bA\|_{\infty}+F_0F_1]\Delta t^2  = c_0\Delta t^2,   \label{error2.3}
\end{align}
 where $c_0$ is the same constant as appeared in the proof of the previous theorem. Combine \eqref{error2.1}, \eqref{error2.2} and \eqref{error2.3} we have
$$\|\bm e_{m+1}\|_{\infty} \le (1+F_1\Delta t)\|\bm e_m\|_{\infty} + c_0\Delta t^2.$$
By the same arguments as in  Theorem \ref{error_estimate_1a}, we obtain
$$\|\bm e_m\|_{\infty}\le C(e^{F_1t_m}-1)\Delta t,$$
where  $C=\frac{c_0}{F_1}>0$ is independent of $\Delta t$.
\end{proof}

\begin{theorem}\label{error_estimate_2}
 Given any fixed terminal time $T>0$ and spatial mesh size $h>0$. Suppose that $f$ satisfies assumptions (A1) and (A3). Assume that $\bm u\in \mathcal{C}([0,T];\mathbb{R}^N)$ is the exact solution of \eqref{semi-discrete} and $\{\bm U_m\}$ is the numerical solution generated by the LRI2 scheme (cf.  Equation~\eqref{LRI2a}). Then,  there exists $C=C(\beta,\delta_0,\omega_0, F_0,F_1,\widetilde{F_1}$, $F_2,\widetilde{F_2}, F_3,\|\bA\|_{\infty})>0$ independent of $\Delta t$ such that the following estimate is true for all $0<\Delta t \le \delta_0\omega_0$:
$$\|\bm U_m-\bm u(t_m)\|_{\infty} \le C(e^{F_4t_m}-1)\Delta t^2, \quad m=0,1,\ldots,M,$$
where $F_4:=F_1+\frac 12\delta_0\omega_0F_3.$
\end{theorem}

\begin{proof}
Let $\bR_3(t_m)$ be the truncation error of the LRI2 scheme \eqref{LRI2a}, we have
\begin{align}
 \bm u(t_{m+1})=&\;e^{\Delta t \bA}\bm u(t_m) + \frac{\Delta t}{2}\left[e^{\Delta t \bA}\bmf(\bm u(t_m))+\bmf(e^{\Delta t \bA} \bm u(t_m)\right] + \frac{\Delta t^2}{2}e^{\Delta t \bA}\frac{\partial \bmf}{\partial \bm u}(\bm u(t_m))\bmf(\bm u(t_m)) + \bR_3(t_m).   \label{truncation3}
\end{align}
This together with \eqref{LRI2a} gives us
\begin{align*}
 \bm e_{m+1}=\;&e^{\Delta t \bA} \bm e_m + \frac{\Delta t}{2}\left[e^{\Delta t \bA}(\bmf(\bm U_m)-\bmf(\bm u(t_m))+\bmf(e^{\Delta t \bA}\bm U_m)-\bmf(e^{\Delta t \bA}\bm u(t_m))\right] \\
 & + \frac{\Delta t^2}{2}e^{\Delta t \bA} \left[\frac{\partial \bmf}{\partial \bm u}(\bm U_m)\bmf(\bm U_m)-\frac{\partial \bmf}{\partial \bm u}(\bm u(t_m))\bmf(\bm u(t_m))\right] - \bR_3(t_m).
\end{align*}
Therefore,
\begin{align*}
\|\bm e_{m+1}\|_{\infty} &\le \|\bm e_m\|_{\infty} +\frac{\Delta t}{2}\left[\|\bmf(\bm U_m)-\bmf(\bm u(t_m)\|_{\infty} + \|\bmf(e^{\Delta t \bA}\bm U_m)-\bmf(e^{\Delta t \bA}\bm u(t_m))\|_{\infty}\right]\\
& +\frac{\Delta t^2}{2}\left\|\frac{\partial \bmf}{\partial \bm u}(\bm U_m)\bmf(\bm U_m)-\frac{\partial \bmf}{\partial \bm u}(\bm u(t_m))\bmf(\bm u(t_m))\right\|_{\infty} +\|\bR_3(t_m)\|_{\infty}.
\end{align*}
Using the similar arguments as in proofs of Theorems \ref{error_estimate_1a} and \ref{error_estimate_1b}, we can easily obtain
\begin{align*}
 \|\bmf(\bm U_m)-\bmf(\bm u(t_m)\|_{\infty}& \le F_1\|\bm U_m-\bm u(t_m)\|_{\infty}= F_1\|\bm e_m\|_{\infty},\\
 \|\bmf(e^{\Delta t \bA}\bm U_m)-\bmf(e^{\Delta t \bA}\bm u(t_m))\|_{\infty}& \le F_1\|e^{\Delta t \bA}\bm U_m-e^{\Delta t \bA}\bm u(t_m)\|_{\infty}\le F_1\|\bm e_m\|_{\infty}.
\end{align*}
Define $\bm l(\bv)=\frac{\partial \bmf}{\partial \bm u}(\bv)\bmf(\bv)$ for $\bv\in\mathbb{R}^N$. It is easy to verify that $\| \frac{\partial \bm l}{\partial \bv}\|_{\infty}\le F_3$ when $\|\bv\|_{\infty}\le \beta$, then 
$$\|\bm l(\bm U_m)-\bm l(\bm u(t_m))\|_{\infty}\le F_3\|\bm U_m-\bm u(t_m)\|_{\infty} = F_3\|\bm e_m\|_{\infty}.$$
Consequently, we have
\begin{align}\label{error3.1}
 \|\bm e_{m+1}\|_{\infty} \le \left(1+F_1\Delta t +\frac 12F_3\Delta t^2\right)\|\bm e_m\|_{\infty}+\|\bR_3(t_m)\|_{\infty}.   
\end{align}
On the other hand,  by comparing \eqref{truncation3} and \eqref{exact1}, for $\bR_3(t_m)$ we arrive at
\begin{align*}
\bR_3(t_m)&=\int_0^{\Delta t} e^{(\Delta t -s)\bA}\bmf(\bm u(t_m+s))\,ds- \frac{\Delta t}{2}\left[e^{\Delta t \bA}\bmf(\bm u(t_m)) +\bmf(e^{\Delta t \bA} \bm u(t_m)\right]-\frac{\Delta t^2}{2}e^{\Delta t \bA}\frac{\partial \bmf}{\partial \bm u}(\bm u(t_m))\bmf(\bm u(t_m))\\
&\hspace{-0.45cm}= \int_0^{\Delta t} \left [e^{(\Delta t -s)\bA} \bmf(\bm u(t_m+s)) -se^{\Delta t \bA}\frac{\partial \bmf}{\partial \bm u}(\bm u(t_m))\bmf(\bm u(t_m))- \left(1-\frac{s}{\Delta t}\right)e^{\Delta t \bA}\bmf(\bm u(t_m))-\frac{s}{\Delta t}\bmf(e^{\Delta t\bA}\bm u(t_m))\right ]\,ds\\
&\hspace{-0.45cm}= \int_0^{\Delta t}(I_1+I_2+I_3+I_4)\,ds,
\end{align*}
where $I_i, i=1,2,3,4$ are defined as follow:
\begin{align*}
I_1&=e^{(\Delta t -s)\bA} \left[\bmf(\bm u(t_m+s))-\bmf(e^{s\bA}\bm u(t_m)+s\bmf(\bm u(t_m)))\right],\\
I_2&=e^{(\Delta t -s)\bA}\bigg[\bmf(e^{s\bA}\bm u(t_m)+s\bmf(\bm u(t_m)))-\left(\bmf(e^{s\bA}\bm u(t_m))+s\frac{\partial \bmf}{\partial \bm u}(e^{s\bA}\bm u(t_m))\bmf(\bm u(t_m))\right)\bigg],\\
I_3&=se^{(\Delta t -s)\bA} \frac{\partial \bmf}{\partial \bm u}(e^{s\bA}\bm u(t_m))\bmf(\bm u(t_m))-se^{\Delta t \bA}\frac{\partial \bmf}{\partial \bm u}(\bm u(t_m))\bmf(\bm u(t_m)),\\
I_4&= e^{(\Delta t -s)\bA} \bmf(e^{s\bA}\bm u(t_m)) - \left(1-\frac{s}{\Delta t}\right)e^{\Delta t \bA}\bmf(\bm u(t_m))-\frac{s}{\Delta t}\bmf(e^{\Delta t\bA}\bm u(t_m)).
\end{align*}
Since $\|\bm u(t_m)\|_{\infty}\le \beta$ and $s\le \Delta t\le \delta_0\omega_0$, we have
\begin{align}\label{over1}
 \|e^{s\bA}\bm u(t_m)+s\bmf(\bm u(t_m))\|_{\infty}\le \|e^{s\bA}\bm u(t_m)\|_{\infty}+s\|\bmf(\bm u(t_m))\|_{\infty}\le \beta+\delta_0\omega_0F_0.  
\end{align}
It is easy to see that when $\bm v, \bm w\in\mathbb{R}^N$ and $\|\bm v\|_{\infty},\|\bm w\|_{\infty}\le \beta + \delta_0\omega_0F_0$, then
$$\|\bmf(\bm v)-\bmf(\bm w)\|_{\infty}\le \widetilde{F_1}\|\bm v-\bm w\|_{\infty}.$$
This, together with the formula \eqref{exact2}, gives us
\begin{align*}
    \|I_1\|_{\infty} &\le \left\|\bmf(\bm u(t_m+s))-\bmf(e^{s\bA}\bm u(t_m)+s\bmf(\bm u(t_m)))\right\|_{\infty}\le \widetilde{F_1}\left\|\bm u(t_m+s)-\left(e^{s\bA}\bm u(t_m)+s\bmf(\bm u(t_m)\right)\right\|_{\infty}\\
    & = \widetilde{F_1}\left\|s\bmf(\bm u(t_m)) - \int_0^s e^{(s-\sigma)\bA}\bmf(\bm u(t_m+\sigma))\,d\sigma\right\|_{\infty}=\widetilde{F_1}\left\| \int_0^s \left [\bmf(\bm u(t_m))- e^{(s-\sigma)\bA}\bmf(\bm u(t_m+\sigma))\right ]\,d\sigma\right\|_{\infty}\\
     & \le \widetilde{F_1} \int_0^s \big [\|\bmf(\bm u(t_m))-e^{s\bA}\bmf(\bm u(t_m))\|_{\infty} + \|e^{s\bA}\bmf(\bm u(t_m))-e^{(s-\sigma)\bA} \bmf(e^{\sigma \bA}\bm u(t_m))\|_{\infty} \\
    &\qquad\qquad +\|e^{(s-\sigma)\bA} \bmf(e^{\sigma \bA}\bm u(t_m))- e^{(s-\sigma)\bA} \bmf(\bm u(t_m+\sigma))\|_{\infty}\big ]\, d\sigma\\
    &=: \widetilde{F_1} \int_0^s (J_1+J_2+J_3) \,d\sigma,
\end{align*}
where $J_1,J_2$ and $J_3$ are given by
\begin{align*}
J_1&=\|\bmf(\bm u(t_m))-e^{s\bA}\bmf(\bm u(t_m))\|_{\infty},\\
J_2&=\|e^{s\bA}\bmf(\bm u(t_m))-e^{(s-\sigma)\bA} \bmf(e^{\sigma \bA}\bm u(t_m))\|_{\infty},\\
J_3&=\|e^{(s-\sigma)\bA} \bmf(e^{\sigma \bA}\bm u(t_m))- e^{(s-\sigma)\bA} \bmf(\bm u(t_m+\sigma))\|_{\infty}.
\end{align*}
By the mean value theorem, we have
$J_1\le F_0\|e^{s\bA}-I\|_{\infty}\le sF_0\|\bA\|_{\infty}.$
Using Lemma \ref{lemma_phi_psi} and the mean value theorem, we obtain
$J_2\le (F_0+\beta F_1)\|\bA\|_{\infty} \sigma.$
Using the formula \eqref{exact2}, we have
\begin{align*}
 J_3&\le F_1\|e^{\sigma \bA}\bm u(t_m)-\bm u(t_m+\sigma)\|_{\infty}=F_1\left\|\int_0^{\sigma}e^{(\sigma-\xi)\bA}\bmf(\bm u(t_m+\xi))d\xi\right\|_{\infty}   \le F_1\int_0^{\sigma} F_0d\xi =F_0F_1\sigma.
\end{align*}
Therefore,
\begin{align}\label{I1est}
 \|I_1\|_{\infty}&\le \widetilde{F_1}\int_0^s(J_1+J_2+J_3)\,d\sigma \le \widetilde{F_1}\int_0^s [F_0\|\bA\|_{\infty}s + (F_0+\beta F_1)\|\bA\|_{\infty}\sigma + F_0F_1\sigma]d\sigma \nonumber \\
 & =\frac 12\widetilde{F_1}[(3F_0+\beta F_1)\|\bA\|_{\infty}+F_0F_1] s^2=:c_1s^2.
\end{align}
By Taylor expansion, there exists $\bm V\in\mathbb{R}^N, \|\bm V\|_{\infty}\le \beta + \delta_0\omega_0 F_0$ (from \eqref{over1}) such that
\begin{align*}
 \bmf(e^{s\bA}\bm u(t_m)+s\bmf(\bm u(t_m)))&=\bmf(e^{s\bA}\bm u(t_m))+s\frac{\partial \bmf}{\partial \bm u}(e^{s\bA}\bm u(t_m))\bmf(\bm u(t_m))+\frac{\partial^2\bmf}{\partial \bm u^2}(\bm V)\frac{{[s\bmf(\bm u(t_m))]^{\odot 2}}}{2}.   
\end{align*}
Thus one gets
\begin{equation}\label{I2est}
\|I_2\|_{\infty}\le \left\|\frac{\partial^2\bmf}{\partial \bm u^2}(\bm V)\frac{[s\bmf(\bm u(t_m))]^{\odot 2}}{2}\right\|_{\infty} \le \frac 12 F_0^2\widetilde{F_2} s^2=:c_2s^2.
\end{equation}
To estimate $I_3$, using Lemma \ref{lemma_phi_psi} and the mean value theorem, we obtain
\begin{equation}\label{I3est}
\|I_3\|_{\infty}=s\|\bm\psi(s)-\bm\psi(0)\|_{\infty}\le F_0(F_1+\beta F_2)\|\bA\|_{\infty}s^2=:  c_3s^2.
\end{equation}
For $I_4$ notice that
$$\|I_4\|_{\infty} = \left\|\bm\lambda(s)-\left(1-\frac{s}{\Delta t}\right)\bm\lambda(0)-\frac{s}{\Delta t}\bm\lambda(\Delta t)\right\|_{\infty}.$$
Using Taylor expansion, there exist $\bm V_1, \bm V_2\in\mathbb{R}^N$ satisfying $$\|\bm V_1\|_{\infty},\|\bm V_2\|_{\infty}\le \max\limits_{s\in[0,\Delta t]}\|\bm\lambda''(s)\|_{\infty}\overset{\text{Lemma \ref{lemma_phi_psi}}}{\le} (F_0+3\beta F_1+\beta^2F_2)\|\bA\|_{\infty}^2$$ subject to
$$\begin{cases}
\bm\lambda(0)=\bm\lambda(s)-s\bm\lambda'(s)+\frac{s^2}{2}\bm V_1,\\[2pt]
\bm\lambda(\Delta t) = \bm\lambda(s) +(\Delta t -s)\bm\lambda'(s) + \frac{(\Delta t -s)^2}{2} \bm V_2.
\end{cases}$$
Thus we have
\begin{align}\label{I4est}
 \|I_4\|_{\infty} &= \left\|\left(1-\frac{s}{\Delta t}\right)\frac{s^2}{2}\bm V_1 + \frac{s}{\Delta t}\frac{(\Delta t -s)^2}{2} \bm V_2\right\|_{\infty}\le \frac 12(F_0+3\beta F_1+\beta^2F_2)\|\bA\|_{\infty}^2(s\Delta t -s^2)=:c_4(s\Delta t -s^2).
\end{align}
Combining the above estimates \eqref{I1est}-\eqref{I4est} for $I_i, i=1,2,3,4$ yields
\begin{align}\label{error3.2}
 \|\bR_3(t_m)\|_{\infty} &\le \int_0^{\Delta t} (c_1+c_2+c_3)s^2+c_4(s\Delta t-s^2) \,ds =\left(\frac{c_1+c_2+c_3}{3}+\frac{c_4}{6}\right)\Delta t^3=:c\Delta t^3.   
\end{align}
From \eqref{error3.1} and \eqref{error3.2}, we get
\begin{align*}
 \|\bm e_{m+1}\|_{\infty}&\le \left(1+F_1\Delta t +\frac 12F_3\Delta t^2\right)\|\bm e_m\|_{\infty} + c\Delta t^3\le \left(1+F_4\Delta t\right) \|\bm e_m\|_{\infty} +c\Delta t^3.
\end{align*}
Therefore,
\begin{align*}
\|\bm e_{m}\|_{\infty}+\frac{c\Delta t^2}{F_4} \le (1+F_4\Delta t)^m\left(\|\bm e_{0}\|_{\infty}+\frac{c\Delta t^2}{F_4}\right).    
\end{align*}
Using the fact that $\bm e_0=0$, we deduce that
\begin{align*}
  \|\bm e_m\|_{\infty}\le \frac{c\Delta t^2}{F_4}[(1+F_4\Delta t)^m-1] \le \frac{c\Delta t^2}{F_4}(e^{F_4t_m}-1)=:C(e^{F_4t_m}-1)\Delta t^2,
\end{align*}
where $C=\frac{c}{F_4} = \frac{1}{F_4}\left(\frac{c_1+c_2+c_3}{3}+\frac{c_4}{6}\right)>0$, and $c_1,c_2,c_3,c_4, F_4$ depend on $\beta,\delta_0$, $\omega_0,F_0$, $F_1,\widetilde{F_1}$, $F_2,\widetilde{F_2}, F_3,$ and $\|\bA\|_{\infty}.$
\end{proof}

\begin{remark}
For the case of $f(u)=u-u^{3}$ (i.e., the Allen-Cahn equation with the      double-well potential), the estimates derived in Theorems~\ref{error_estimate_1a},  \ref{error_estimate_1b} and \ref{error_estimate_2} could be made simpler. In particular,  for the LRI1a and LRI1b schemes,  there exists $C=C(\|\bA\|_{\infty})>0$ such that the following inequality holds for all $0<\Delta t\le \frac 12$:
$$\|\bm U_m-\bm u(t_m)\|_{\infty} \le C(e^{2t_m}-1)\Delta t, \quad m=0,1,\ldots,M.$$
For the LRI2 scheme,  there exists $C=C(\|\bA\|_{\infty})>0$ so that for all $0<\Delta t\le \frac 35$ we have:
$$\|\bm U_m-\bm u(t_m)\|_{\infty} \le C(e^{\frac{16}{5}t_m}-1)\Delta t^2, \quad m=0,1,\ldots,M.$$
\end{remark}

\section{Numerical experiments}\label{section_numerical}

We now verify numerically the convergence rates in time, MBP preservation and energy stability of the proposed LRI schemes (LRI1a \eqref{LRI1a},  LRI1b~\eqref{LRI1b}, and LRI2 \eqref{LRI2a}) analyzed in the previous sections.  To illustrate the advantages of the LRI schemes,  we compare their performance with the ETD schemes \cite{Review} of orders 1 and 2 (i.e., ETD1 and ETDRK2) in terms of time discretization errors as the interfacial parameter $\varepsilon$ approaches zero.  Two test cases in the two dimensional domain $\Omega=(-0.5,0.5)^{2}$ are presented: Test case 1 having  a traveling wave solution and Test case 2 modeling the coarsen dynamics.  For the latter,  both double-well potential and Flory-Huggins potential functions will be considered in the numerical simulations.  \Rv{All numerical results are obtained via a fast implementation of the LRI and ETD methods based on matrix decomposition and Discrete Fourier Transform (DFT) as proposed in \cite{fEIF}.  The code is implemented in MATLAB on a MacBook Pro with the M1 Pro chip and 32 GB memory}.

\subsection{Test case 1: Traveling wave}

 To investigate the convergence in time of the proposed schemes,  we use the benchmark test~\cite{twave} where a traveling wave solution of the Allen-Cahn equation~\eqref{Allen-Cahn} with $f(u)=u-u^{3}$ is considered. In particular,  if the initial data is given by
\begin{equation}\label{test1:ic}
 u_{0}(x,y)=\frac{1}{2}\left (1-\tanh\left (\frac{x}{2\sqrt{2}\varepsilon}\right )\right ), \; (x,y) \in \overline{\Omega},
 \end{equation}
then the Allen-Cahn equation in the whole space has the traveling wave solution 
\begin{equation}  \label{test1:wavesoln}
u(x,y,t)=\frac{1}{2}\left (1-\tanh\left (\frac{x-st}{2\sqrt{2}\varepsilon}\right )\right ), \; (x,y) \in \overline{\Omega},
\end{equation}
where $s=\sfrac{3\varepsilon}{\sqrt{2}}$ is the speed of the traveling wave.  We impose a homogeneous Neumann boundary condition on the whole boundary $\partial \Omega$, so that \eqref{test1:wavesoln} can be used as an {\em approximate exact solution} (for $\varepsilon \ll 1$) on $\overline{\Omega} \times [0,T]$.  We take $\varepsilon \in \left \{0.02, 0.01, 0.005\right \}$, and set the ending time $T=\sfrac{1}{4s}$. 
With a fixed mesh size $h=1/2048$, we compute the numerical solutions obtained by the first-order methods (ETD1, LRI1a and LRI1b) and the second-order methods (ETDRK2 and LRI2) with decreasing time step sizes.  The $L^{2}$ and $L^{\infty}$ errors at time $t=T$  and their convergence rates of the first-order schemes are reported in Tables~\ref{tab:t1order1L2} and \ref{tab:t1order1Linf}, respectively.  The corresponding results for the second-order schemes are shown in Table~\ref{tab:t1order2}.  

\begin{table}[!htbp]\small
    \caption{[Test case 1] $L^{2}$ errors and convergence rates in time by the first-order ETD and LRI  schemes.} \label{tab:t1order1L2} 
\renewcommand{\arraystretch}{1}
    \centering
    \begin{tabular}{|*{8}{c|}}
    \hline
      \multirow{2}{*}{$\varepsilon$} &   \multirow{2}{*}{$\Delta t$} & \multicolumn{2}{c|}{ETD1} & \multicolumn{2}{c|}{LRI1a} & \multicolumn{2}{c|}{LRI1b}  \\ \cline{3-8}
        ~ &~  & $L^{2}$ Error & Rate & $L^{2}$ Error & Rate & $L^{2}$ Error & Rate  \\ \hline
        \multirow{6}{*}{0.02} & ${T}/{32}$ & 4.24e-2 & -- & 3.35e-2 & -- & 2.86e-2 & --  \\ 
        ~ &${T}/{64}$ & 2.24e-2 & 0.92 & 1.75e-2 & 0.93 & 1.50e-2 & 0.93  \\ 
         ~ &${T}/{128}$ &1.15e-2 & 0.96 & 8.98e-3 & 0.97 & 7.66e-3 & 0.97  \\ 
        ~ &${T}/{256}$ & 5.83e-3 & 0.98 & 4.55e-3 & 0.98 & 3.87e-3 & 0.98  \\ 
        ~ &${T}/{512}$ & 2.93e-3 & 0.99 & 2.29e-3 & 0.99 & 1.95e-3 & 0.99  \\  
        ~ &${T}/{1024}$ &1.47e-3 & 1.00 & 1.15e-3 & 1.00 & 9.77e-4 & 1.00  \\  \hline
        \multirow{6}{*}{0.01} & ${T}/{32}$ & 1.11e-1 & -- & 8.51e-2 & -- & 7.89e-2 & --  \\ 
        ~ &${T}/{64}$ & 6.36e-2 & 0.81 & 4.73e-2 & 0.85 & 4.37e-2 & 0.85  \\ 
        ~ &${T}/{128}$ & 3.40e-2 & 0.90 & 2.50e-2 & 0.92 & 2.30e-2 & 0.93  \\ 
        ~ &${T}/{256}$ & 1.76e-2 & 0.95 & 1.28e-2 & 0.96 & 1.18e-2 & 0.96  \\   
        ~ &${T}/{512}$ & 8.91e-3 & 0.98 & 6.48e-3 & 0.98 & 5.96e-3 & 0.98  \\   
        ~ &${T}/{1024}$ & 4.48e-3 & 0.99 & 3.25e-3 & 0.99 & 2.99e-3 & 0.99 \\  \hline
        \multirow{6}{*}{0.005} &${T}/{32}$ & 2.08e-1 & -- & 1.70e-1 & -- & 1.66e-1 & --  \\ 
        ~ &${T}/{64}$ & 1.45e-1 & 0.52 & 1.12e-1 & 0.61 & 1.08e-1 & 0.61  \\  
        ~ &${T}/{128}$ & 8.93e-2 & 0.70 & 6.51e-2 & 0.78 & 6.28e-2 & 0.79  \\ 
        ~ &${T}/{256}$ & 4.93e-2 & 0.86 & 3.49e-2 & 0.90 & 3.36e-2 & 0.90  \\ 
        ~ &${T}/{512}$ & 2.57e-2 & 0.94 & 1.80e-2 & 0.96 & 1.73e-2 & 0.96  \\ 
        ~ &${T}/{1024}$ & 1.30e-2 & 0.98 & 9.07e-3 & 0.99 & 8.69e-3 & 0.99 \\ \hline
    \end{tabular}
\end{table}

\begin{table}[!htbp]\small
\caption{[Test case 1] $L^{\infty}$ errors and convergence rates in time by the first-order ETD and LRI schemes.} \label{tab:t1order1Linf} 
\renewcommand{\arraystretch}{1}
    \centering
    \begin{tabular}{|*{8}{c|}}
    \hline
      \multirow{2}{*}{$\varepsilon$} &   \multirow{2}{*}{$\Delta t$} & \multicolumn{2}{c|}{ETD1} & \multicolumn{2}{c|}{LRI1a} & \multicolumn{2}{c|}{LRI1b}  \\ \cline{3-8}
        ~ &~  & $L^{\infty}$ Error & Rate & $L^{\infty}$ Error & Rate & $L^{\infty}$ Error & Rate  \\ \hline
        \multirow{6}{*}{0.02} & ${T}/{32}$ & 1.61e-1 & -- & 1.28e-1 & -- & 1.12e-1 & --  \\  
        ~ &${T}/{64}$ & 8.49e-2 & 0.93 & 6.70e-2 & 0.93 & 5.83e-2 & 0.93  \\  
        ~ &${T}/{128}$ &4.35e-2 & 0.96 & 3.43e-2 & 0.97 & 2.97e-2 & 0.97  \\  
        ~ &${T}/{256}$ & 2.20e-2 & 0.98 & 1.74e-2 & 0.98 & 1.50e-2 & 0.98  \\ 
        ~ &${T}/{512}$ & 1.11e-2 & 0.99 & 8.73e-3 & 0.99 & 7.54e-3 & 0.99  \\  
        ~ &${T}/{1024}$ &5.55e-3 & 1.00 & 4.38e-3 & 1.00 & 3.78e-3 & 1.00  \\  \hline
        \multirow{6}{*}{0.01} &${T}/{32}$ & 5.51e-1 & -- & 4.30e-1 & -- & 4.11e-1 & --  \\  
        ~ &${T}/{64}$ & 3.30e-1 & 0.74 & 2.47e-1 & 0.80 & 2.32e-1 & 0.82  \\   
        ~ &${T}/{128}$ & 1.78e-1 & 0.89 & 1.31e-1 & 0.91 & 1.22e-1 & 0.92  \\   
        ~ &${T}/{256}$ & 9.22e-2 & 0.95 & 6.76e-2 & 0.96 & 6.27e-2 & 0.97  \\
        ~ &${T}/{512}$ & 4.67e-2 & 0.98 & 3.42e-2 & 0.98 & 3.16e-2 & 0.99  \\ 
        ~ &${T}/{1024}$ & 2.35e-2 & 0.99 & 1.72e-2 & 0.99 & 1.59e-2 & 1.00 \\  \hline
        \multirow{6}{*}{0.005} &${T}/{32}$ & 9.68e-1 & -- & 9.06e-1 & -- & 9.04e-1 & --  \\ 
        ~ &${T}/{64}$ & 8.49e-1 & 0.19 & 7.13e-1 & 0.34 & 7.07e-1 & 0.35  \\   
        ~ &${T}/{128}$ & 6.05e-1 & 0.49 & 4.57e-1 & 0.64 & 4.48e-1 & 0.66  \\  
        ~ &${T}/{256}$ & 3.55e-1 & 0.77 & 2.54e-1 & 0.85 & 2.47e-1 & 0.86  \\  
        ~ &${T}/{512}$ & 1.88e-1 & 0.92 & 1.32e-1 & 0.94 & 1.28e-1 & 0.95  \\ 
        ~ &${T}/{1024}$ & 9.57e-2 & 0.97 & 6.69e-2 & 0.98 & 6.43e-2 & 0.99 \\ \hline
    \end{tabular}
\end{table}

\begin{table}[!htbp]\small
  \caption{[Test case 1] Errors and convergence rates in time by the second ETD and LRI  schemes.} \label{tab:t1order2}
\renewcommand{\arraystretch}{1}
    \centering
    \begin{tabular}{|*{10}{c|}}
    \hline
      \multirow{2}{*}{$\varepsilon$} &   \multirow{3}{*}{$\Delta t$} & \multicolumn{4}{c|}{ETDRK2} & \multicolumn{4}{c|}{LRI2}  \\ \cline{3-10}
        ~ &~  & $L^{2}$  & \multirow{2}{*}{Rate} & $L^{\infty}$ & \multirow{2}{*}{Rate} & $L^{2}$  & \multirow{2}{*}{Rate} & $L^{\infty}$  & \multirow{2}{*}{Rate} \\ 
        ~ &~  & Error  & ~ & Error & ~ & Error  & ~ & Error  & ~ \\ \hline
         \multirow{5}{*}{0.02}  & ${T}/{16}$  & 8.92e-3 & -- & 3.08e-2 & -- & 6.34e-3 & -- & 2.57e-2 & --  \\  
         & ${T}/{32}$  & 2.50e-3 & 1.83 & 8.83e-3 & 1.80 & 1.80e-3  & 1.82 & 7.24e-3 & 1.83  \\     
        ~ &${T}/{64}$ & 6.65e-4 & 1.91 & 2.36e-3 & 1.90 & 4.80e-4 & 1.91 & 1.92e-3 & 1.91  \\   
        ~ &${T}/{128}$ & 1.71e-4 & 1.95 & 6.10e-4 & 1.95 & 1.24e-4 & 1.95 & 4.94e-4 & 1.96  \\ 
        ~ &${T}/{256}$& 4.57e-5 & 1.91 & 1.53e-4 & 2.00 & 3.48e-5 & 1.84 & 1.43e-4 & 1.79  \\ 
        ~ &${T}/{512}$ & 1.95e-5 & 1.23 & 1.43e-4 & 0.10 & 1.81e-5 & 0.94 & 1.43e-4 & 0.00 \\ \hline 
         \multirow{5}{*}{0.01}  & ${T}/{16}$  & 3.75e-2 & -- & 1.80e-1 & -- & 2.97e-2 & -- & 1.62e-1 & --  \\ 
         & ${T}/{32}$  & 1.18e-2 & 1.67 & 5.92e-2 & 1.61 & 9.44e-3 & 1.65 & 5.14e-2 & 1.66  \\ 
        ~ &${T}/{64}$ & 3.35e-3 & 1.81 & 1.70e-2 & 1.80 & 2.68e-3 & 1.82 & 1.45e-2 & 1.83  \\ 
        ~ &${T}/{128}$ & 8.87e-4 & 1.92 & 4.52e-3 & 1.91 & 7.05e-4 & 1.92 & 3.80e-3 & 1.93  \\ 
        ~ &${T}/{256}$& 2.19e-4 & 2.02 & 1.12e-3 & 2.02 & 1.71e-4 & 2.04 & 9.28e-4 & 2.04  \\ 
        ~ &${T}/{512}$ & 4.45e-5 & 2.30 & 2.30e-4 & 2.28 & 3.25e-5 & 2.40 & 1.81e-4 & 2.36 \\  \hline 
        \multirow{5}{*}{0.005}  & ${T}/{16}$  & 2.74e-1 & -- & 7.06e-1 & -- & NaN & -- & NaN & --  \\ 
        & ${T}/{32}$ & 4.87e-2 & 2.49 & 3.38e-1 & 1.06 & 4.23e-2 & -- & 3.13e-1 & --   \\ 
        ~ &${T}/{64}$ & 1.67e-2 & 1.55 & 1.17e-1 & 1.53 & 1.36e-2 & 1.64 & 1.02e-1 & 1.62  \\ 
        ~ &${T}/{128}$ & 6.74e-3 & 1.31 & 2.77e-2 & 2.08 & 3.78e-3 & 1.84 & 2.82e-2 & 1.85  \\ 
        ~ &${T}/{256}$ & 3.05e-2 & -2.18 & 7.06e-2 & -1.35 & 9.24e-4 & 2.03 & 6.90e-3 & 2.03  \\ 
        ~ &${T}/{512}$ & 3.61e-2 & -0.24 & 7.65e-2 & -0.12 & 1.54e-4 & 2.59 & 1.17e-3 & 2.57    \\  \hline
    \end{tabular}
  \end{table}

For the first-order schemes, we observe expected first-order temporal convergence rates in both $L^{2}$ and $L^{\infty}$ norms with various values of $\varepsilon$. In addition, the errors by two LRI1 schemes are always smaller than those by ETD1; the performance of LRI1a and LRI1b is quite similar as $\varepsilon$ decreases.  For the second-order schemes, we first notice that the errors are much smaller compared to the first-order schemes with the same time step sizes,  especially when $\varepsilon$ is close to zero.  Importantly,  second-order temporal convergence rates of LRI2 are confirmed even for small $\varepsilon$. For $\varepsilon=0.02$,  the errors in space are dominant when $\Delta t$ is sufficiently small, that is why the convergence rates deteriorate after several refinements in time.  Moreover, we remark that the errors are computed using the approximate exact solution~\eqref{test1:wavesoln} which is accurate when $\varepsilon$ is close enough to zero. 
Unlike LRI2,  ETDRK2 fails to converge when $\varepsilon=0.005$; a similar behavior was also observed for the Navier-Stokes equations with the classical exponential integrator when the viscosity coefficient is close to zero~\cite{LRINS}.  Furthermore,  as for the first-order schemes,  the errors by ETDRK2 are always bigger than those by LRI2 with the same $\Delta t$ for all considered values of $\varepsilon$; this phenomenon becomes much more significant when $\varepsilon$ gets smaller. 

\Rv{We report in Table~\ref{tab:t1CPUtime} the running times (in seconds) of LRI and ETD methods with fixed $h=1/2048$ and varying time step sizes. The running times are similar for different values of $\varepsilon$, thus we only show the results when $\varepsilon=0.01$.  We observe that for the first-order schemes,  LRI1a and LRI1b have similar running times and are faster than ETD1.  For the second-order methods,  LRI2 also outperforms ETDRK2 regarding CPU time.  Thus, we conclude that LRI1 and LRI2 schemes are more effective than ETD1 and ETDRK2 methods, respectively,  in terms of both accuracy and computational cost. }

%
\begin{table}[!htbp] \small
\setlength{\extrarowheight}{2pt}
    \caption{\Rv{[Test case 1] Running times (in seconds) of the first and second-order LRI and ETD schemes with $h = 1/2048$.}} \label{tab:t1CPUtime} 
\renewcommand{\arraystretch}{1}
    \centering
    \Rv{
    \begin{tabular}{|*{6}{c|}}
    \hline
    $\Delta t$ & ETD1 & LRI1a & LRI1b & ETDRK2 & LRI2  \\ \hline
        ${T}/{32}$ &  29.77& 24.12&24.13 & 41.51&35.14\\ 
        ${T}/{64}$ & 47.08& 35.73& 35.79& 70.57&57.95\\ 
        ${T}/{128}$ & 82.13& 59.42& 59.51 & 129.20&103.76\\ 
        ${T}/{256}$ & 152.24& 107.50& 106.96 & 246.61&194.71\\ 
        ${T}/{512}$ &  296.41& 204.15& 203.84& 480.35&379.02\\ \hline
    \end{tabular}}
\end{table}

%

\subsection{Test case 2: Coarsening dynamics}

Next we validate  MBP preservation and energy dissipation of the proposed LRIs.  Toward that end,  we simulate the process of the coarsening dynamics by considering a random initial configuration.  
The model problem is given by the Allen-Cahn equation~\eqref{Allen-Cahn} with periodic boundary conditions.  
 We consider two choices for the nonlinear function $f(u)$:
\begin{equation} \label{eq:nonlinfu}
f(u)=f_{1}(u)=u-u^{3}, \quad \text{or} \quad f(u)=f_{2}(u)=\frac{\theta}{2} \text{ln}\frac{1-u}{1+u} + \theta_{c} u,
\end{equation}
corresponding to the      double-well and Flory-Huggins potential cases,  respectively.  For the latter, the parameters are $\theta=0.8$ and $\theta_{c}=1.6$ as in Remark~\ref{rmk:MBPtimestep}. 
We set $\varepsilon=0.01$, and fix the spatial mesh $h=1/1024$. The initial data is generated by random numbers on each mesh point, the range of these numbers will be specified for each nonlinear function considered.

When $f=f_{1}$ (double-well potential case) in \eqref{eq:nonlinfu},  we run the simulation with the random initial data ranging from $-1$ to $1$ and a uniform time step of size $\Delta t= 0.5$ for LRI1a and LRI1b, and $\Delta t=0.6$ for LRI2.  Note that these are the upper bounds of the time steps predicted by our theoretical results (cf. Remark~\ref{rmk:MBPtimestep}).  Evolutions of the supremum norm and the energy of the numerical solutions are depicted in Figure~\ref{fig:t2dwMBP}.  It can be seen that  the numerical solutions obtained are all bounded by $1$ in absolute value, and the associated energies are dissipative over a long time horizon. Figure~\ref{fig:t2dwsoln} shows the snapshots of the numerical solution at $t=5, 10, 20,40,80,$ and $120$,  generated by using the LRI2 scheme with a smaller time step $\Delta t=0.1$; when $\Delta t=0.05$,  we  obtain very similar results for the evolution of the numerical solution.

\begin{figure}[!htbp]
    \centering
    \begin{subfigure}[b]{0.43\textwidth}
        \includegraphics[width=\textwidth]{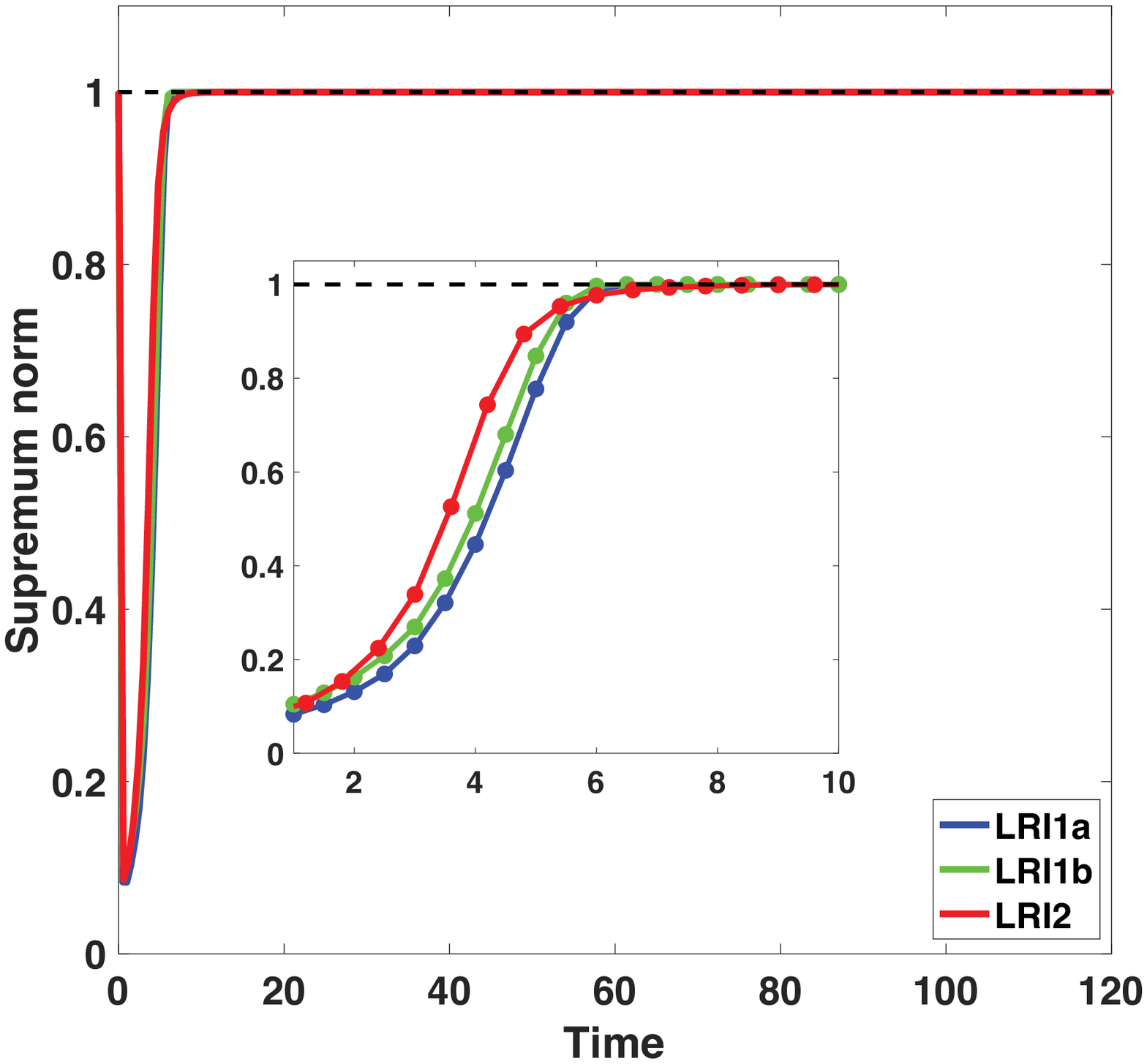}
    \end{subfigure} \hspace{0.2cm}
    \begin{subfigure}[b]{0.44\textwidth}
        \includegraphics[width=\textwidth]{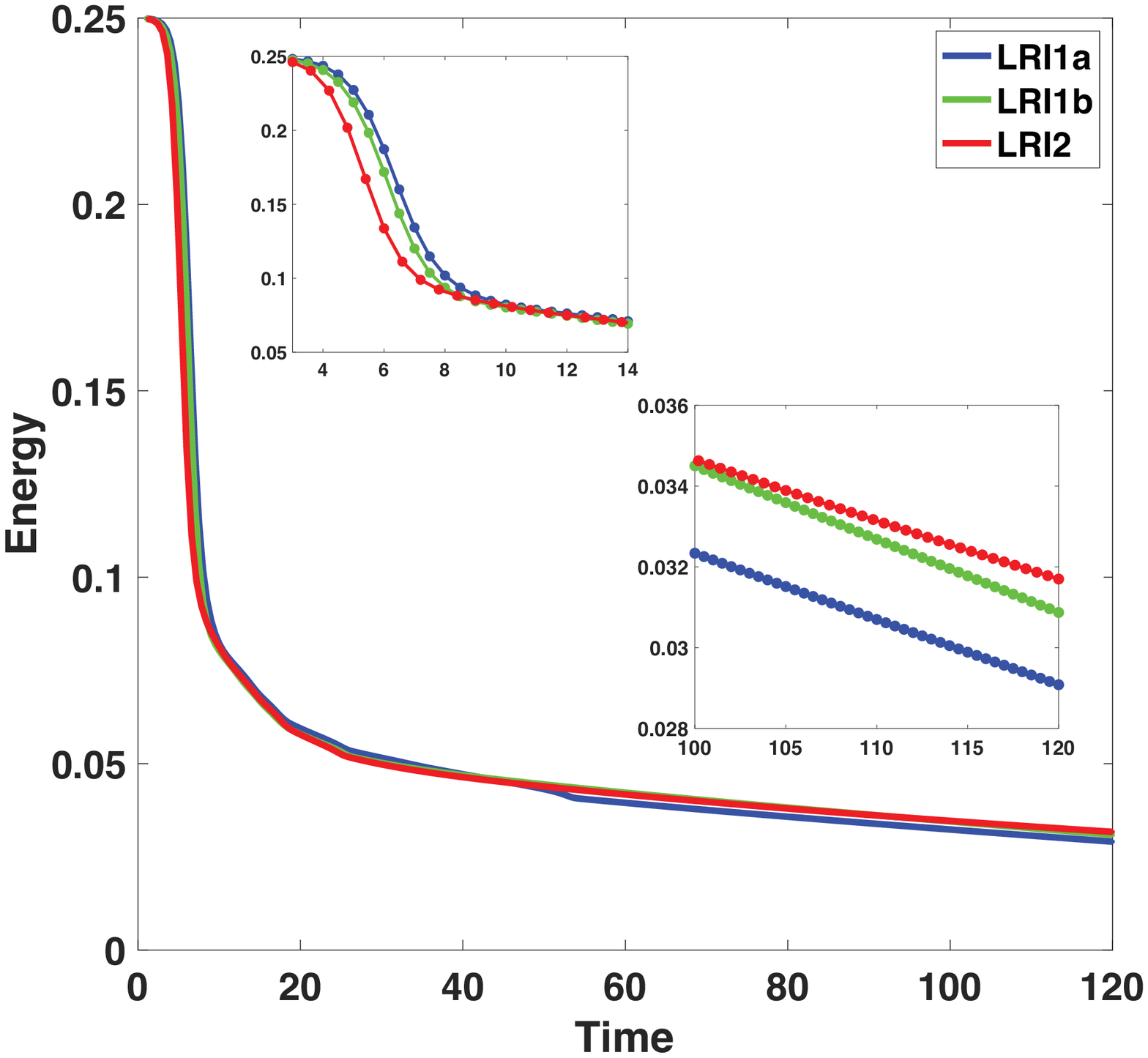}
    \end{subfigure}
    \caption{[Test case 2 with the double-well potential $f_1$] Evolutions of the supremum norm (left) and the energy (right) of the numerical solutions by the LRI1a and LRI1b schemes with $\Delta t =$ 0.5 and by the LRI2 scheme with $\Delta t =$ 0.6. }\label{fig:t2dwMBP} 
\end{figure}
\begin{figure}[!htbp]
    \centerline{
        \includegraphics[width=0.36\textwidth]{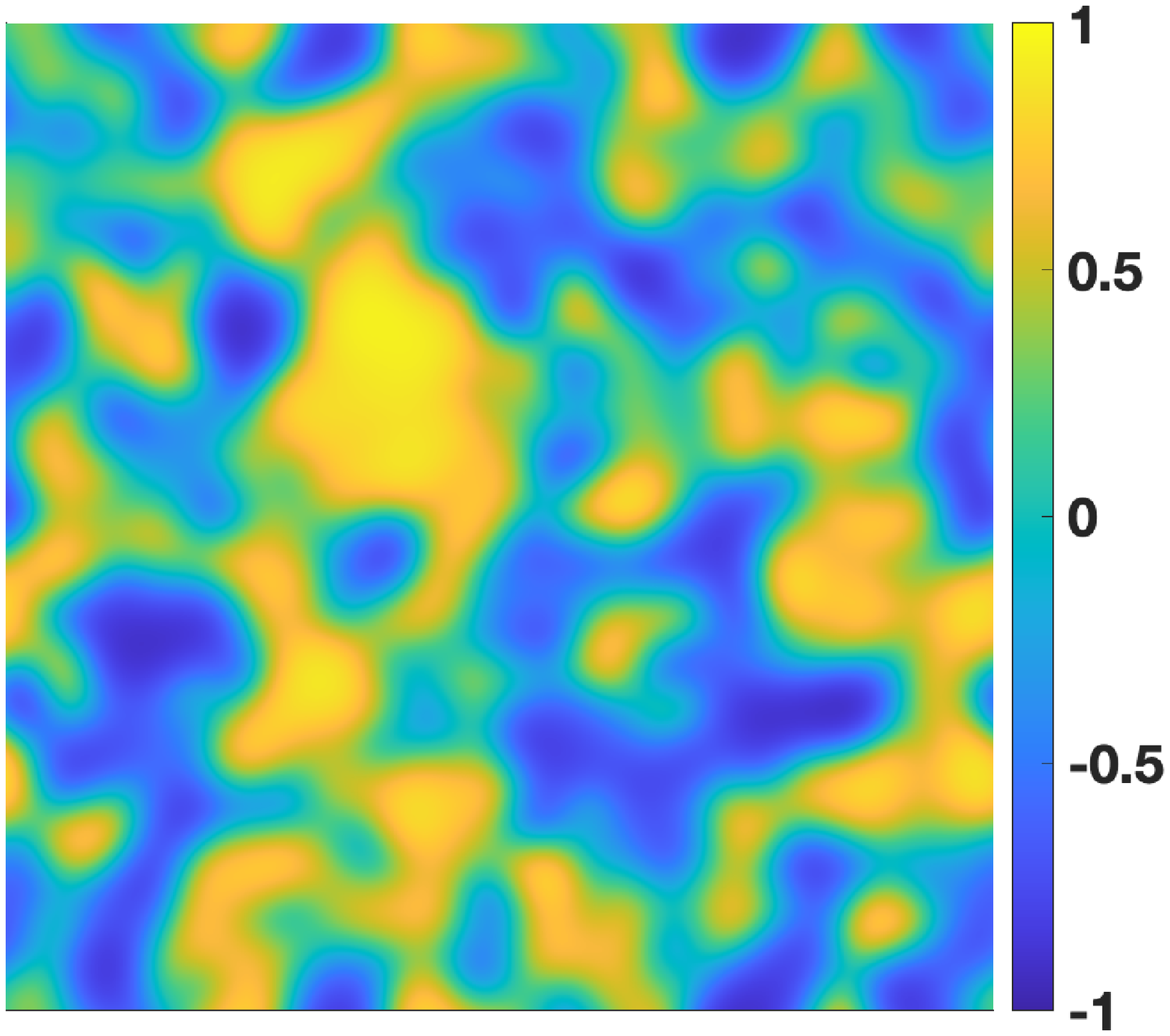}\hspace{-0.5cm}
        \includegraphics[width=0.36\textwidth]{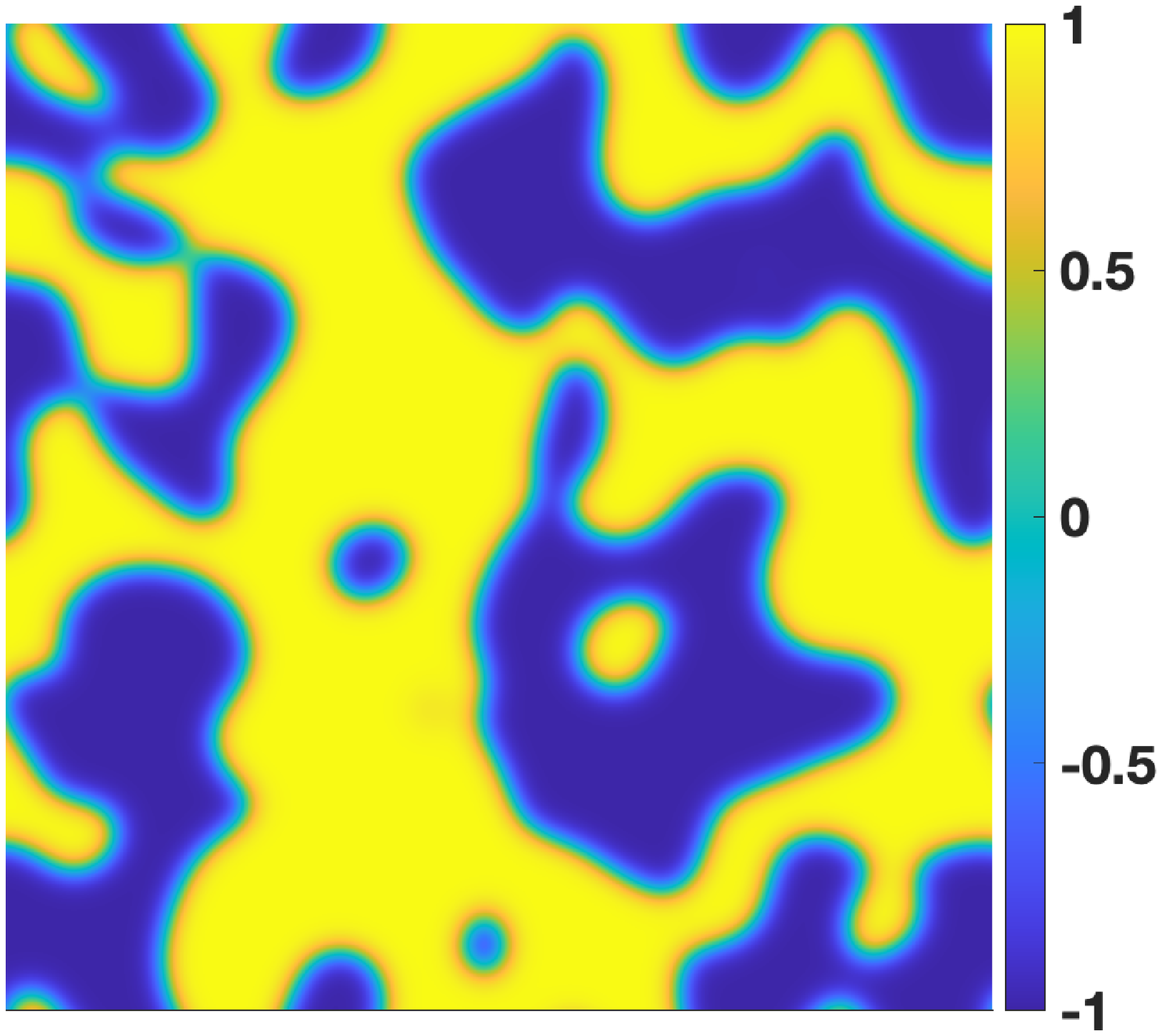}\hspace{-0.5cm}
        \includegraphics[width=0.36\textwidth]{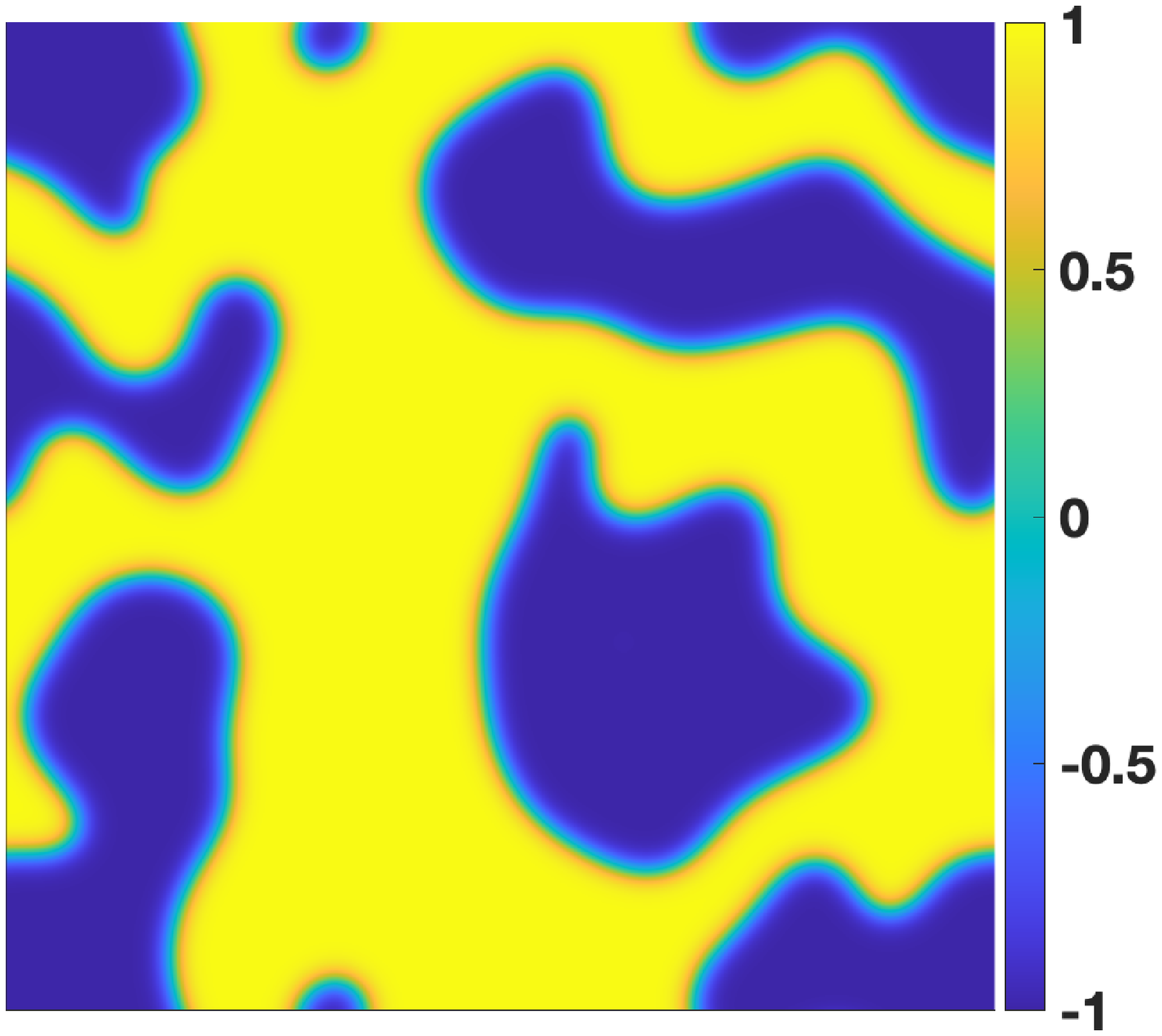}}
    \centerline{
        \includegraphics[width=0.36\textwidth]{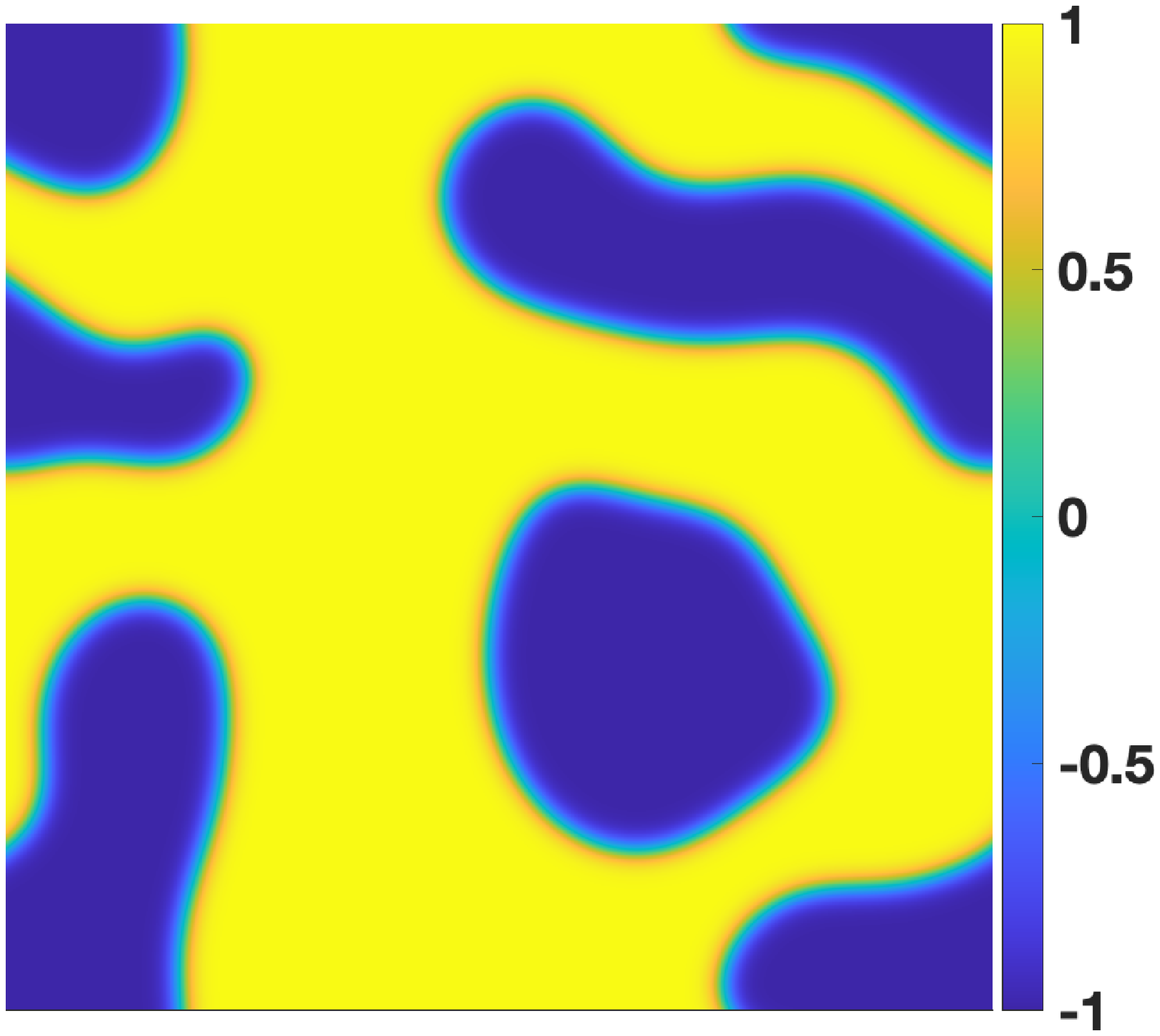}\hspace{-0.5cm}
        \includegraphics[width=0.36\textwidth]{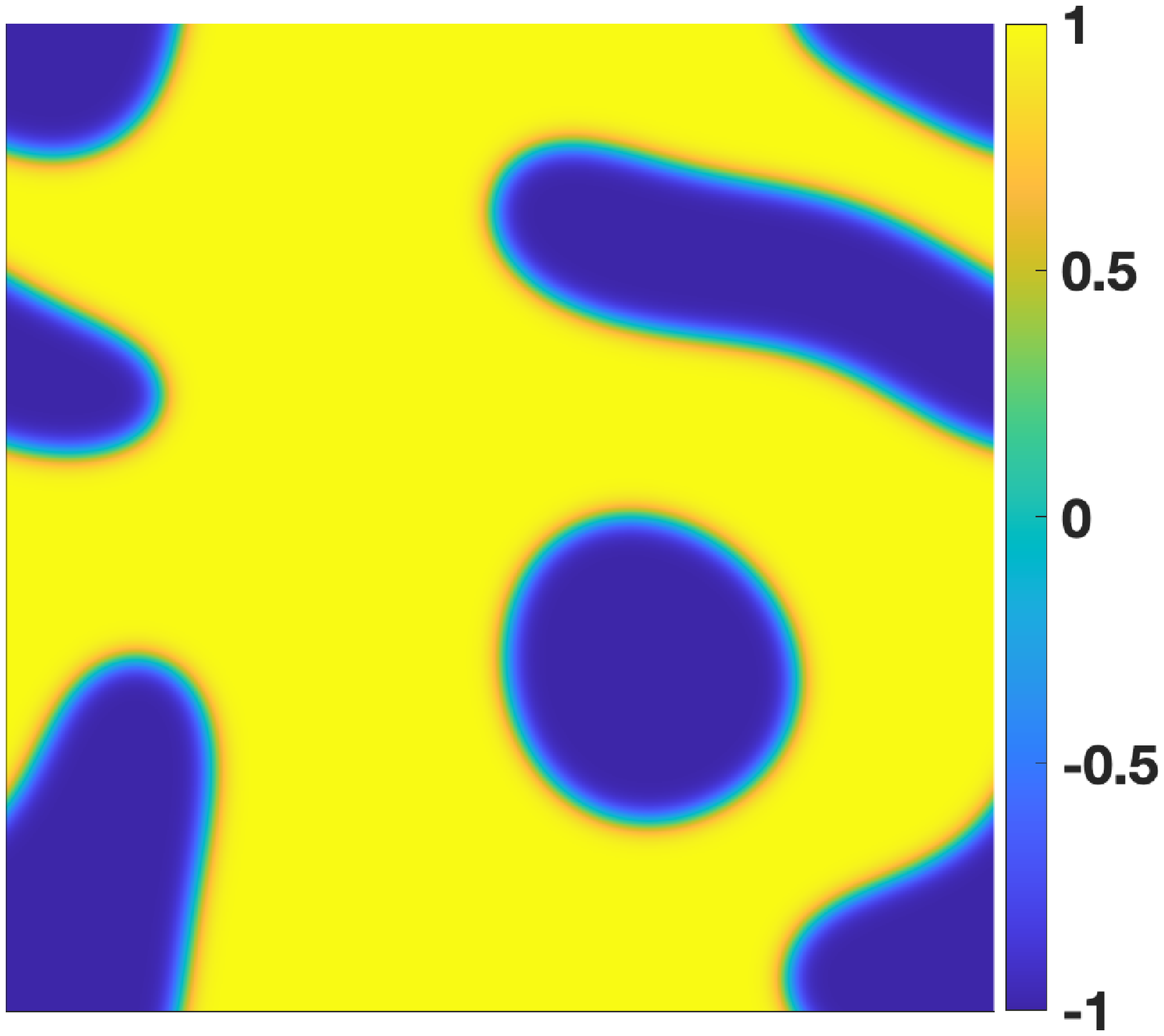}\hspace{-0.5cm}
        \includegraphics[width=0.36\textwidth]{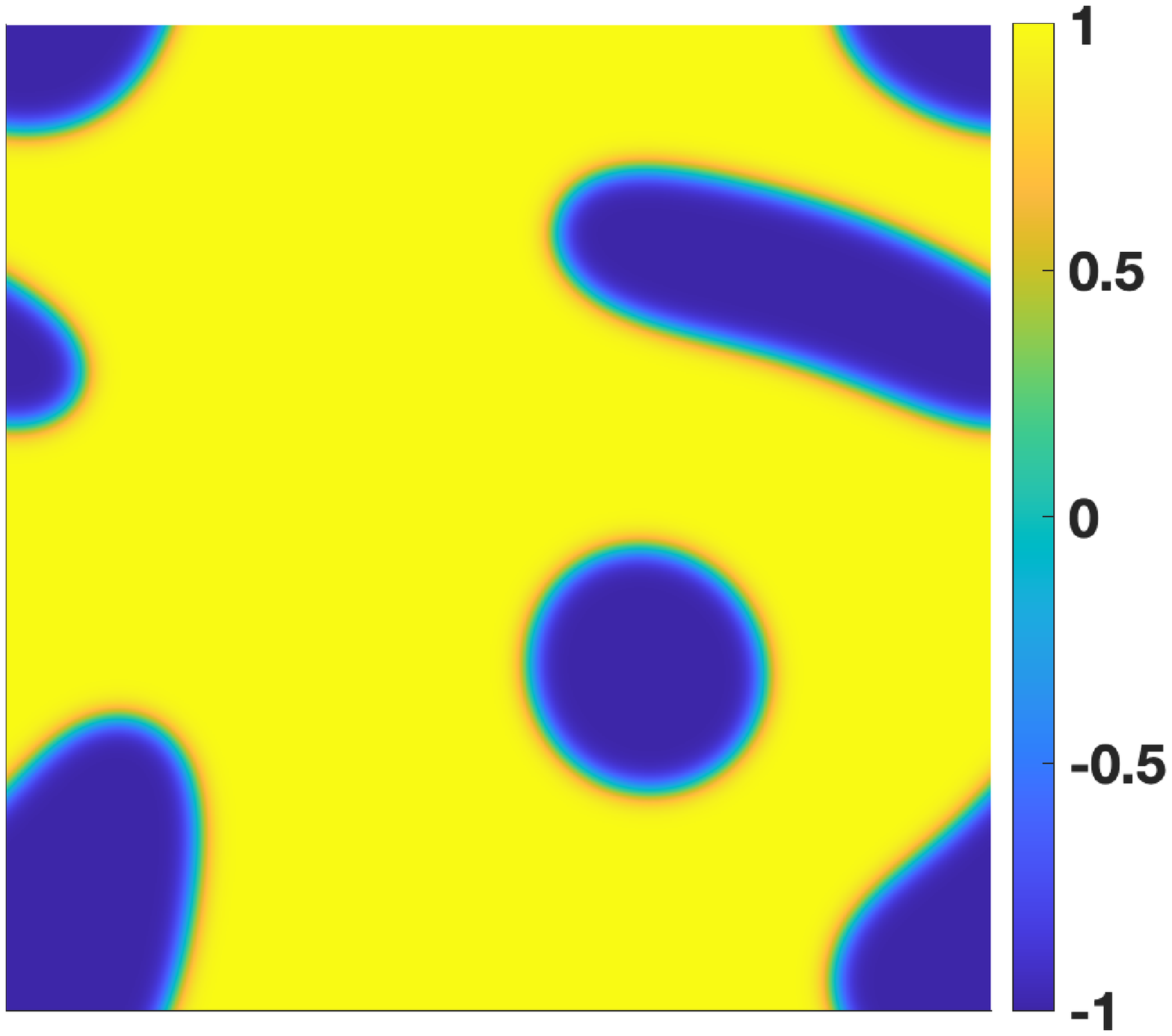}}
    \caption{[Test case 2 with the  the double-well potential $f_1$] Snapshots of the numerical solution at $t=$ 5, 10, 20, 40, 80 and 120 (from left to right and top to bottom) produced by the LRI2 scheme with $\Delta t$ = 0.1.}\label{fig:t2dwsoln} \vspace{-0.4cm}
\end{figure}

\begin{figure}[!htbp]
    \centering
    \begin{subfigure}[b]{0.43\textwidth}
        \includegraphics[width=\textwidth]{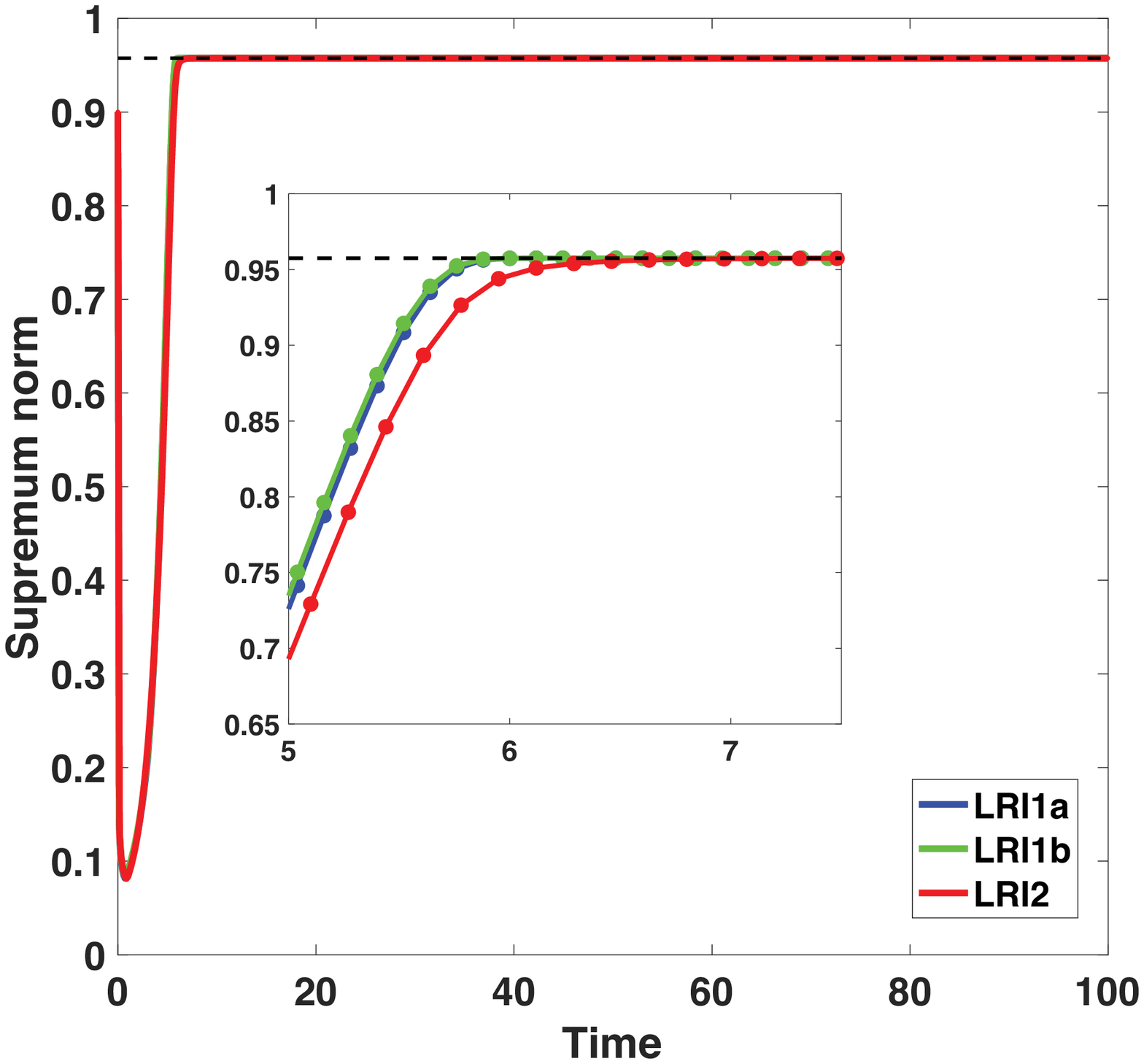}
    \end{subfigure} \hspace{0.2cm}
    \begin{subfigure}[b]{0.44\textwidth}
        \includegraphics[width=\textwidth]{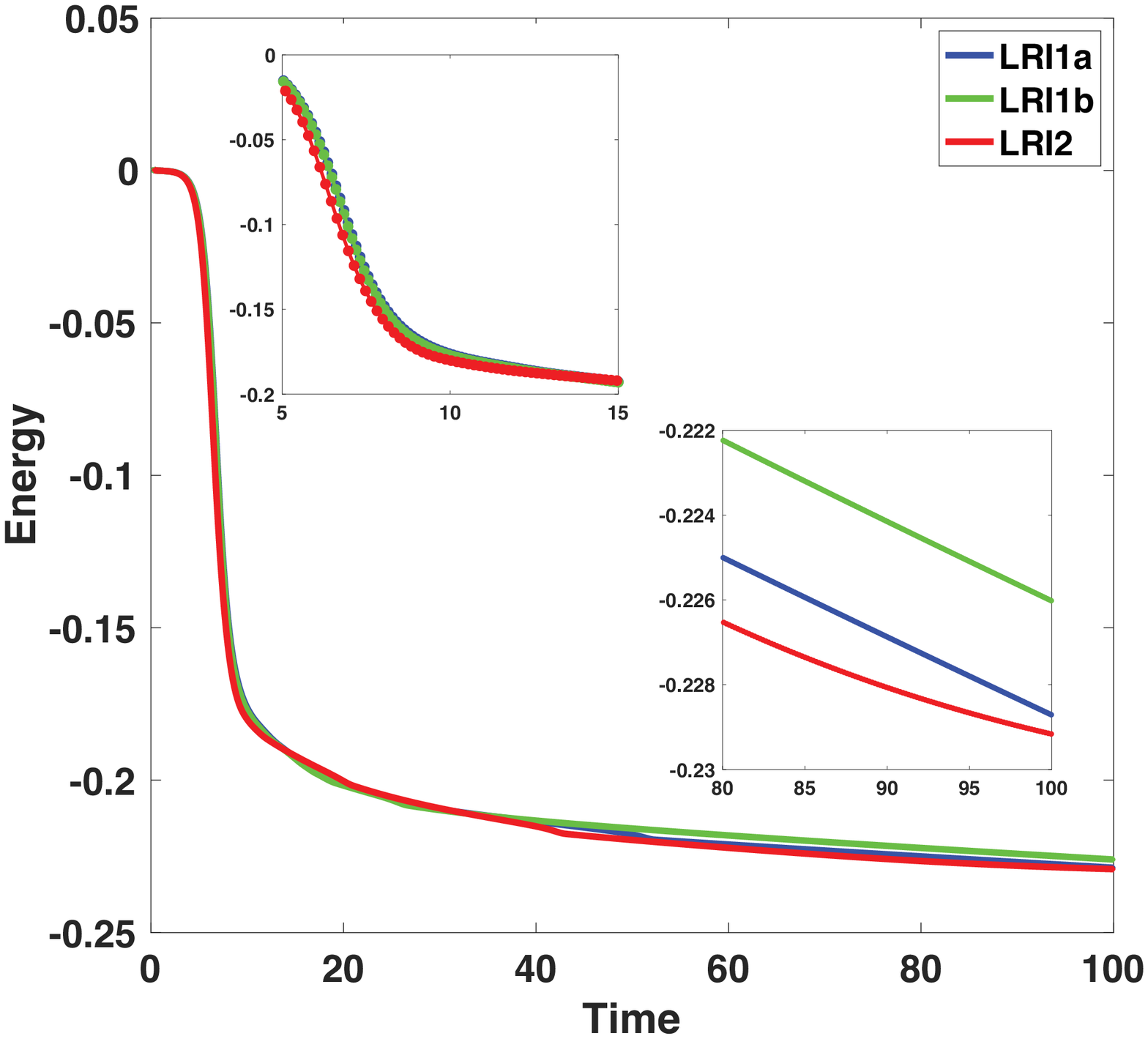}
    \end{subfigure}
    \caption{[Test case 2 with the Flory-Huggins potential $f_2$] Evolutions of the supremum norm (left) and the energy (right) of the numerical solutions by the  LRI1a and LRI1b schemes with $\Delta t =$ 0.12 and by the LRI2 scheme with $\Delta t =$ 0.17. }\label{fig:t2fhMBP}  \vspace{-0.4cm}
\end{figure}

When $f=f_{2}$ (Flory-Huggins potential case) in \eqref{eq:nonlinfu}, we use the random initial data ranging from $-0.9$ to $0.9$. A uniform time step of of size $\Delta t= 0.12$ is used for LRI1a and LRI1b, and $\Delta t=0.17$ for LRI2, which are close to the upper bounds estimated in Remark~\ref{rmk:MBPtimestep}.  Evolutions of the supremum norm and the energy of the numerical solutions  are plotted in Figures~\ref{fig:t2fhMBP}. We again observe that the LRI schemes preserve MBP very well with the chosen time steps,  i.e.,  the supremum norm of numerical solutions does not exceed $\beta\approx 0.9575$,  and the energy monotonically decays along the time.

\section{Conclusion}\label{section_conclusion}

In this work, we have studied some low regularity integrators for a class of semilinear parabolic equations with MBP. We introduce the first-order (LRI1a and LRI1b) and the second-order (LRI2) schemes, which preserve the MBP under very reasonable time step sizes.  Through the case of Allen-Cahn equation, we elaborate that the range of the time step can be enlarged when the second-order scheme is used.  We show that the discrete energy functional is bounded uniformly for the proposed LRI schemes by using the MBP and the mean value theorem. Then we also 
 prove  the optimal temporal  error estimates (under a fixed spatial mesh size) of the LRI numerical solutions under the assumption that the exact space-discrete solution is continuous, which is a much weaker regularity  requirement compared to other methods such as ETD or IFRK. 
Finally, some numerical experiments are performed and the results confirm the theoretical analysis and demonstrate the advantages of using the proposed LRI schemes.  Our ongoing work includes \Rv{ the fully discrete error analysis for the proposed LRI schemes when both the mesh size and the time step size simultaneously go to zero,} and the application of LRIs to other types of phase-field equations.  In addition,  higher-order LRI schemes will also be considered,  with low regularity requirements for the exact solution and high regularity assumptions on the nonlinear term. 

\section*{Acknowledgements}
T.-T.-P. Hoang's work is partially supported by U.S. National Science Foundation under grant number DMS-2041884.  L. Ju's work is partially supported by U.S. National Science Foundation under grant number DMS-2109633. K. Schratz's work is partially funded by European Research Council (ERC) under the European Union’s Horizon 2020 research and innovation programme (grant agreement No. 850941).

\section*{Declarations} 

{\bf Conflict of interest} The authors declare that there is no conflict of interest.

\end{document}